\documentclass[12pt]{article}

\usepackage{times}

\usepackage{amssymb}
\usepackage{amsmath,amsfonts}
\usepackage{float}
\usepackage{mathtools}
\usepackage{mathrsfs}

\newtheorem{definition}{Def.}
\newtheorem{theorem}{Theorem}
\newtheorem{lemma}{Lemma}
\newtheorem{proof}{Proof}
\newtheorem{remark}{Remark}

\usepackage{amssymb}
\usepackage{authblk}
\begin{document}

%\begin{frontmatter}

%% Title, authors and addresses

%% use the tnoteref command within \title for footnotes;
%% use the tnotetext command for theassociated footnote;
%% use the fnref command within \author or \address for footnotes;
%% use the fntext command for theassociated footnote;
%% use the corref command within \author for corresponding author footnotes;
%% use the cortext command for theassociated footnote;
%% use the ead command for the email address,
%% and the form \ead[url] for the home page:
%% \title{Title\tnoteref{label1}}
%% \tnotetext[label1]{}
%% \author{Name\corref{cor1}\fnref{label2}}
%% \ead{email address}
%% \ead[url]{home page}
%% \fntext[label2]{}
%% \cortext[cor1]{}
%% \address{Address\fnref{label3}}
%% \fntext[label3]{}

\title{\large \textbf{BIE and BEM approach for the mixed Dirichlet-Robin boundary value problem for the nonlinear Darcy-Forchheimer-Brinkman system}}

%% use optional labels to link authors explicitly to addresses:
%% \author[label1,label2]{}
%% \address[label1]{}
%% \address[label2]{}

\author{Robert Gutt}

\affil{\small Faculty of Mathematics and Computer Science,
Babe\c{s}-Bolyai University, 1 M. Kog\u{a}l\-niceanu Str., 400084
Cluj-Napoca, Romania}
\date{}
\maketitle

\begin{abstract}
The purpose of this paper is the mathematical analysis of the weak solution of the mixed Dirichlet-Robin boundary value problem for the nonlinear  Darcy-Forch\-heimer-Brinkman system in a bounded, two-dimensional Lipschitz domain, and the application of the corresponding results to the study of the lid-driven flow problem of an incompressible viscous fluid located in a square cavity filled with a porous medium. First we obtain a well-posedness result for the linear Brinkman system with Dirichlet-Neumann boundary conditions, employing a variational approach for the corresponding boundary integral equations. The result is extended afterwards to the Poisson problem for the Brinkman system and to Dirichlet-Robin boundary conditions.  Further, we study the nonlinear Darcy-Forchheimer-Brinkman boundary value problem of Dirichlet and Robin type. Using the Dual Reciprocity Boundary Element Method (DRBEM), we numerically investigate a special Dirichlet-Robin boundary value problem associated to the nonlinear Darcy-Forchheimer-Brinkman system, i.e., the lid-driven cavity flow problem. The physical properties of such a flow are discussed by the geometry of the streamlines of the fluid flow for different Reynolds and Darcy numbers. Moreover, an additional sliding parameter is imposed on the moving wall and we show its importance by segmenting the upper lid into two opposite moving segments.

\end{abstract}
%% MSC codes here, in the form: \MSC code \sep code
%% or \MSC[2008] code \sep code (2000 is the default)

%\end{frontmatter}

%% \linenumbers

%% main text
\section*{Introduction}
\label{Intro} 
Convection phenomena of a viscous, incompressible fluid through a porous media is described by the nonlinear Darcy-Forchheimer Brinkman system, when the inertia of such a fluid is not negligible. Various practical applications arise corresponding to boundary value problems on Lipschitz domains related to this system, due to the evolution of modern technologies, such as the use of geothermal energy in geology, secondary oil recuperation in petroleum engineering, disposal of radioactive wastes the list could go on.

Many different methods have been employed in the mathematical analysis of elliptic boundary value problems, such as variational approaches, potential theory, parametrix (Levi function) integral methods. Moreover, various numerical studies are concerned with such systems going from central differences, finite element methods, meshless methods and also boundary element methods.

We mention the work of Fabes, Kenig and Verchota \cite{Fa-Ke}, which studies the Dirichlet problem for the Stokes system by reducing it to the analysis of related boundary integral equation and the work of Kohr and Wendland \cite{K-W}, who studied the mixed Dirichlet-Neumann problem for the Stokes system on Lipschitz domains employing a variational approach for the related boundary integral equations. Mitrea and Wright \cite{M-W} obtained several well-posedness results for the Stokes system in Lipschitz domains with data in $L^{p}$, Sobolev and Besov spaces.

The Navier-Stokes equation received also a great attention over the last decades. Choe and Kim \cite{Choe-Kim} analyzed the Dirichlet problem for the Navier-Stokes system on bounded Lipschitz domains with connected boundary. By using a double-layer formulation, Russo and Tartaglione \cite{Russo-Tartaglione-2} studied the Robin problem for the Stokes system in bounded or exterior Lipschitz domains in $\mathbb{R}^{3}$ (see also \cite{Russo-Tartaglione-3}  for the Oseen system and \cite{Russo-Tartaglione-4} for the Navier-Stokes system). Also, many studies are concerned with mixed boundary value problems (see, e.g., \cite{M-M}, \cite{Brown}, \cite{Chkadua}). Gatica, Hsiao, Meddahi and Sayas in \cite{Sayas} treated a dual-mixed formulation for an exterior Stokes problem in $\mathbb{R}^{3}$. Let us also mention that, the well-posedness result for the mixed Dirichlet-Robin problem for the Brinkman system in a creased Lipschitz domain $\mathfrak{D} \subset \mathbb{R}^{n} \ (n \geq 3)$ with boundary data in $L^2$-based Sobolev spaces has been recently obtained in \cite[Theorem 6.1]{K-L-W2} (see also \cite[Theorem 7.1]{GuttMMA} for Dirichlet-Neumann boundary conditions and \cite{K-L-W4} for Robin boundary conditions).

Layer potential analysis and fixed point theorems have been used by Kohr, Lanza de Cristoforis and Wendland in \cite{kohrpotanal1} and \cite{K-L-W4} to obtain existence results for nonlinear Neumann-transmission problems for the Stokes and the Brinkman systems on Lipschitz domains in $ \mathbb{R}^{n} $. Also, Dirichlet-transmission problems have been treated in \cite{Kohr-ZAMM}, \cite{Fericean}. Moreover, Kohr, Lanza de Cristoforis, Mikhailov, Wendland \cite{K-L-M-W} obtained well-posedness results for the transmission problems between a bounded and unbounded exterior domain. The  mixed Dirichlet-Neumann problem for the semilinear Darcy-Forchheimer-Brinkman system in a creased Lipschitz domain $\mathfrak{D} \subset \mathbb{R}^{3} $ with boundary data in $L^p$-based Besov spaces has been recently studied in \cite[Theorem 7.1]{GuttMMA}.

One of many model problems which connect the mathematical theory with computational methods is the lid-driven cavity flow problem, due to the reason that it reflects the physical properties of the fluid flow. Steady solutions for the lid-driven problem for the Navier-Stokes system for Reynolds numbers up to $ Re = 20000$ have been presented by Erturk, Corke and Gokcol \cite{Erturk-Gokcol}. Sahin and Owens \cite{Sahin} used the finite volume method to obtain stability solutions on the driven cavity flow. Koseff and Street managed to obtain some experimental results in \cite{Koseff}, \cite{Koseff2}. Yang, Xue and Mathias \cite{Yang} have analysed the flow of viscous fluids in porous media and the dependency of the evolution of recirculating cells on the cavity depth. 

This paper analysis the weak solution of the mixed Dirichlet-Robin boundary value problem for the nonlinear Darcy-Forchheimer-Brinkman system in a bounded, two-dimensional Lipschitz domain. First, we obtain a well-posedness result for the linear Brinkman system with Dirichlet-Neumann boundary conditions, continuing the work of Kohr and Wendland \cite{K-W} for the Stokes system. The result is extended afterwards to the Poisson problem for the Brinkman system and to Dirichlet-Robin boundary conditions, using the Newtonian potentials and the linearity of the solution operator.  Further, we study the nonlinear Darcy-Forchheimer-Brinkman boundary value problem of Dirichlet and Robin type. 

The second part is concerned with numerical simulations of the lid-driven cavity flow problem associated to such a problem, using the Dual Reciprocity Boundary Element Method (DRBEM). The physical properties of such a flow are discussed by the geometry of the streamlines of the fluid flow for different Reynolds and Darcy numbers. Moreover, an additional sliding parameter is imposed on the moving wall and we show its importance by segmenting the upper moving lid into two segments. This work continues the numerical and theoretical studies started in \cite{Gutt} for the nonlinear Darcy-Forchheimer-Brinkman system with Dirichlet or mixed Dirichlet-Robin type boundary conditions.

%\newpage

\section{Preliminaries}

\begin{definition}
\label{LipschitzDomain2D}
Let ${\mathfrak{D}_{+}}\subset \mathbb{R}^{2}$ be an open, connected and bounded set. We say that ${\mathfrak{D}_{+}}$ is a \textit{Lipschitz domain}, if there exists a constant $M>0$ {such} that for each $x\in \Gamma$ there exist {a coordinate system in $\mathbb{R}^{2}$ (isometric to the canonical one)}, $(x',x'')\in \mathbb{R} \times \mathbb{R}$, a constant $r>0$, a cylinder $ {\mathcal C}_{r}(x) \!:=\! \left\{(y',y''):|y'\! -\! x'| \!< \!r, |y''\! -\! x''|  \!< \!2Mr \right\}$, and a Lipschitz function $\phi : \mathbb{R}\rightarrow \mathbb{R}$ with $\| \nabla \phi \|_{L^{\infty}(\mathbb{R})} \leq M$, such that
\begin{equation}
\begin{split}
&{\mathcal C}_{r}(x)\cap {\mathfrak{D}} = \{(y',y''):y'' > \phi(y')\} \cap {\mathcal C}_{r}(x),
\\
&{\mathcal C}_{r}(x)\cap \Gamma = \{(y',y''): y''=\phi(y')\}\cap {\mathcal C}_{r}(x).
\end{split}
\end{equation}
(see, e.g.,\cite[Definition 3.28]{McLean}, \cite[Section 2]{Ott-2013}).
\end{definition}

Let us define \textit{a dissection of the boundary} for the mixed boundary value problem as in \cite[p. 99]{McLean}, \cite[Section 2]{Ott-2013}.
\begin{definition}
\label{Domain}
Consider a bounded Lipschitz domain $\mathfrak{D} := \mathfrak{D}_{+} \subset \mathbb{R}^{2} $ with connected boundary $\Gamma$, which is partitioned into nonempty subsets $\Gamma_{D}$, $\Lambda$ and $\Gamma_{N}$, such that $\Gamma = \Gamma_{D} \cup \Lambda \cup \Gamma_{N}$. Moreover, we assume that $\Gamma_{D}$ and $\Gamma_{N}$ are disjoint, relatively open subset of $\Gamma$, having $\Lambda$ as their common boundary points. For each $x \in \Lambda$, we require that there exists a coordinate system $(x',x'')$, a coordinate cylinder $ {\mathcal C}_{r}(x)$ centered in $x$, a Lipschitz function $\phi$ as in \eqref{LipschitzDomain2D} and a constant $M_{1}$ such that 
\begin{equation}
\label{Domain1}
\begin{split}
&{\mathcal C}_{r}(x)\cap \Gamma_{D} = \{(y',y''): y' > M_{1},\ y''= \phi(y')\} \cap {\mathcal C}_{r}(x),
\\
&{\mathcal C}_{r}(x)\cap \Gamma_{N} = \{(y',y''): y' < M_{1},\ y''=\phi(y')\}\cap {\mathcal C}_{r}(x),
\end{split}
\end{equation}
(see also  \cite{Ott-2015} for a more special decomposition of the boundary).
\end{definition}

Hence, in the sequel, one can assume that ${\rm meas}(\Gamma_{D}) > 0$ and ${\rm meas}(\Gamma_{N}) > 0$.

Let $\mathfrak{D} = \mathfrak{D}_{+}$ be a bounded Lipschitz domain and let $\mathfrak{D}_{-} := \mathbb{R}^{2} \backslash \overline{\mathfrak{D}}$.  Let $\nu$ be the outward unit normal to $\partial \mathfrak{D} = \Gamma$, which exists almost everywhere on $ \Gamma $. For fixed $\kappa = \kappa(\Gamma) > 1 $, sufficiently large, \textit{the non-tangential maximal operator} is defined as (see, e.g., \cite[Section 2.3]{M-W}, \cite[Section 2]{K-L-W2})
\begin{equation}
\label{2.1.1}
\mathscr{N}(\textbf{u})(x) := \{\sup{|\textbf{u}(y)|: y \in \gamma_{\pm}(x)}, \ x \in \Gamma\},
\end{equation}
for arbitrary $\textbf{u}:\mathfrak{D}_{\pm} \rightarrow \mathbb{R}^{2}$, where
\begin{equation}
\gamma_{\pm} (x) := \{y \in \mathfrak{D}_{\pm} : \rm{dist} (x,y) < \kappa \rm{ dist}(y,\Gamma),\  x \in \Gamma\},
\end{equation}
are \textit{nontangential approach regions} located in $\mathfrak{D}_{+}$ and $\mathfrak{D}_{-}$ respectively. Moreover, we consider \textit{the nontangential boundary traces}  $\rm{Tr}^{\pm}$ on $\Gamma$ of a function $\textbf{u}$ defined in $\mathfrak{D}_{\pm}$, as
\begin{equation}
(\rm{Tr}^{\pm}\textbf{u} )(x) := \lim_{\gamma_{\pm} \ni y \rightarrow x} \textbf{u}(y), \qquad x \in \Gamma,
\end{equation}
\begin{equation}
\rm{Tr}^{\pm} : C^{0}(\overline{\mathfrak{D}}_{\pm}) \rightarrow C^{0}(\Gamma), \qquad \rm{Tr}^{\pm} \textbf{u} = \textbf{u}|_{\Gamma}.
\end{equation}
In the sequel, we will use the superscripts $ \pm $ only when we want to point out which domain is involved, else their meaning will be understood from the context.

We denote by $L^{2}(\mathbb{R}^{2})$ the space (equivalence class) of measurable, square integrable function on $\mathbb{R}^{2}$. For $s \in \mathbb{R}$, the $L^{2}-$based Sobolev (Bessel potential) spaces are defined by
\begin{align}
&H^{s}(\mathbb{R}^{2}) = \{ (\mathbb{I} - \Delta)^{-\frac{s}{2}}f : f \in L^{2}(\mathbb{R}^{2})\} =  \{ \mathcal{F}^{-1}(1 + |\zeta|)^{-\frac{s}{2}} \mathcal{F}f : f \in L^{2}(\mathbb{R}^{2})\} \nonumber \\ \nonumber
&H^{s}(\mathbb{R}^{2}, \mathbb{R}^{2}) = \{ {\bf f} = (f_{1}, f_{2}) : f_{j} \in H^{s}(\mathbb{R}^{2}), \ j = 1,2 \},
\end{align}
where $\mathcal{F}$ is the Fourier transform defined on the space of tempered distributions (see, e.g., \cite[Chapter 4]{H-W}).

The space $H^{s}(\mathfrak{D})$ is defined by the restriction to $\mathfrak{D}$ from $H^{s}(\mathbb{R}^{2})$ and $\widetilde{H^{s}}(\mathfrak{D})$  is the space of functions from $H^{s}(\mathbb{R}^{2})$ with compact support in $\mathfrak{D}$. Moreover, the following duality relations hold (see, e.g., \cite[Prop. 2.9]{Je-Ke}, \cite[4.14]{M-T})
\begin{equation}
(H^{s}(\mathfrak{D}))' = \widetilde{H}^{-s}(\mathfrak{D}) \quad (\widetilde{H}^{s}(\mathfrak{D}))' = H^{-s}(\mathfrak{D}).
\end{equation}

For $ s \in [0,1]$, the Sobolev space $H^{s}(\Gamma)$ on the boundary $\Gamma$ can be defined by using the space $H^{s}(\mathbb{R})$, a partition of unity and a pull-back of the local parametrization of $\Gamma$. Let $S \in \{\Gamma_{D}, \Gamma_{N}\} $ be a part of $\Gamma$ as in definitions \ref{Domain1}. Then we define the following spaces
\begin{align}
H^{s}(S) &= \left\{{\bf f}|_{S} : {\bf f} \in H^{s}(\Gamma)   \right\}, \quad
\widetilde{H}^{s}(S) = \left\{{\bf f} \in H^{s}(\Gamma): {\rm supp}\ {\bf f} \subset \overline{S}  \right\},  \\
%H^{s}(S, \mathbb{R}^{2}) &= \{ \textbf{f} = (f_{1}, f_{2}) : f_{j} \in H^{s}(S), \ j = 1,2 \}, \quad
%\widetilde{H}^{s}(S, \mathbb{R}^{2}) = \{ \textbf{f} = (f_{1}, f_{2}) : f_{j} \in \widetilde{H}^{s}(S), \ j = 1,2 \} \nonumber \\
H^{-s}(S) &= \left\{{\bf f}|_{S} : {\bf f} \in H^{-s}(\Gamma)   \right\}, \quad
\widetilde{H}^{-s}(S) = \left\{{\bf f} \in H^{-s}(\Gamma): {\rm supp}\ {\bf f} \subset \overline{S}  \right\}, \nonumber 
%H^{-s}(S, \mathbb{R}^{2}) &= \{ \textbf{f} = (f_{1}, f_{2}) : f_{j} \in H^{-s}(S), \ j = 1,2 \}, \quad
%\widetilde{H}^{-s}(S, \mathbb{R}^{2}) = \{ \textbf{f} = (f_{1}, f_{2}) : f_{j} \in \widetilde{H}^{-s}(S), \ j = 1,2 \} \nonumber
\end{align}
where $(\cdot )|_{S}$ denotes the restriction operator from $\Gamma$ to $S$. The spaces  $H^{\pm s}(S, \mathbb{R}^{2})$ and $\widetilde{H}^{\pm s}(S, \mathbb{R}^{2})$ are defined as vector-valued functions or distributions whose components belong to the spaces $H^{\pm s}(S)$ and $\widetilde{H}^{\pm s}(S)$, respectively.

Moreover, we have the following duality pairing between the spaces $H^{-\frac{1}{2}}(S, \mathbb{R}^{2})$ and $\widetilde{H}^{-\frac{1}{2}}(S, \mathbb{R}^{2})$ and the spaces $\widetilde{H}^{\frac{1}{2}}(S, \mathbb{R}^{2})$ and $H^{\frac{1}{2}}(S, \mathbb{R}^{2})$ respectively, i.e., (see, e.g.,  \cite[4.14]{M-T})
\begin{align}
H^{-\frac{1}{2}}(S,\mathbb{R}^{2}) = \left(\widetilde{H}^{-\frac{1}{2}}(S,\mathbb{R}^{2})\right)' \quad \widetilde{H}^{\frac{1}{2}}(S,\mathbb{R}^{2}) = \left(H^{-\frac{1}{2}}(S,\mathbb{R}^{2})\right)'.
\end{align}

Finally, let us introduce the following subspace. Let $\nu$ be the outward unit normal to the Lipschitz domain $\mathfrak{D}$, which exists almost everywhere on $\Gamma$ with respect to the surface measure on $\Gamma$. Then we define the subspace $H^{\frac{1}{2}}_{\nu}(\Gamma, \mathbb{R}^{2})$ of $H^{\frac{1}{2}}(\Gamma, \mathbb{R}^{2})$ as
\begin{equation}
\label{OrthSubspace}
H^{\frac{1}{2}}_{\nu}(\Gamma, \mathbb{R}^{2}) := \{ {\bf f} \in H^{\frac{1}{2}}(\Gamma, \mathbb{R}^{2}): \langle {\bf f}, \nu \rangle = 0 \}.
\end{equation}

For more details of Sobolev spaces on Lipschitz domains and boundaries, we refer the reader to \cite{M-W, H-W, M-T-Poisson}.

The following lemma introduces an important result related to the trace operator (see, e.g., \cite{Co}, \cite[Prop. 3.3]{Je-Ke} \cite[Theorem 2.5.2]{M-W}, \cite[Lemma 2.6]{Mikh}).
\begin{lemma}
\label{lem 1.5}
Let $ \mathfrak{D} \subset \mathbb{R}^{2} $ be a bounded Lipschitz domain with connected boundary $ \Gamma $. Then there exists a linear and continuous operator, called the trace operator $ \rm{Tr} : H^{1}(\mathfrak{D}, \mathbb{R}^{2})  \rightarrow H^{\frac{1}{2}}(\Gamma, \mathbb{R}^{2}) $,
such that
\begin{equation}
\rm{Tr} \ {\bf f} = {\bf f}|_{\Gamma}, \quad \forall \ {\bf f} \ \in C^{\infty} (\overline{\mathfrak{D}}, \mathbb{R}^{2})
\end{equation}
Moreover, the trace operator is  surjective and has a continuous, bounded right inverse $ \rm{Tr}^{-1}  : H^{\frac{1}{2}}(\Gamma, \mathbb{R}^{2})  \rightarrow H^{1}(\mathfrak{D}, \mathbb{R}^{2}) $.
\end{lemma}

The conormal derivative or boundary traction for the Stokes system has been introduced by Mitrea and Wright in the general case of Sobolev spaces defined on Lipschitz domains \cite[Theorem 10.4.1]{M-W}. We recall this notion in the case of the Brinkman system with data in Sobolev spaces $ H^{1}(\mathfrak{D}) $ (see, e.g., \cite[Lemma 2.3]{K-L-W2}, \cite[Lemma 2.4]{kohrpotanal1}). We also refer the reader to \cite[Lemma 3.2]{Co}, \cite[Definition 3.1]{Mikh} and to \cite[Lemma 2.5]{Kohr2016} for the Brinkman systems on compact Riemannian manifolds.

For $\alpha > 0$,  the normalized Brinkman system in a bounded domain $\mathfrak{D}$ is given by the following system
 \begin{equation}
\label{L-operator}
\mathcal{L}_{\alpha}({\bf u},p):= \Delta {\bf u}-\alpha \textbf{u} - \nabla p = {\bf f}, \quad \mbox{div} \ \textbf{u} = 0 \ \ \mbox{in} \ \mathfrak{D},
\end{equation}
where $ \textbf{u} $ and $ p $ are the desired velocity and pressure fields of a viscous incompressible fluid flow (see \cite{kohrpotanal} for a detailed description of the Brinkman operator). 

The space $ H^{1}(\mathfrak{D}, \mathcal{L}_{\alpha}) $ is defined by (see, e.g., \cite[Lemma 2.3]{K-L-W2})
\begin{align}
\label{spciudat}
H^{1}(\mathfrak{D}, \mathcal{L}_{\alpha}) := &\left\{( {\bf u},p;{\bf f}) \in H^{1}(\mathfrak{D}, \mathbb{R}^{2}) \times  L^{2}(\mathfrak{D})  \times  \widetilde{H}^{-1}(\mathfrak{D}, \mathbb{R}^{2}) \right. \\
& \left.\mathcal{L}_{\alpha}({\bf u},p)  =  {\bf f}|_{\mathfrak{D}} \ \mbox{and} \ \mbox{div} \ {\bf u} = 0 \ \mbox{in} \ \mathfrak{D} \right\}. \nonumber
\end{align}

\begin{lemma} 
\label{lem 1.6}
Let $ \mathfrak{D}_{+} = \mathfrak{D} \subset \mathbb{R}^{2}$ be a bounded Lipschitz domain with boundary $ \Gamma $ and let $\mathfrak{D}_{-} =  \mathbb{R}^{2}\setminus \overline{\mathfrak{D}}_{+}$.  Let $ \alpha \ge 0 $ and recall the space defined in (\ref{spciudat}). Then \textit{the conormal derivative operator}
\begin{align}
\partial_{\nu;\alpha} &: H^{1}(\mathfrak{D}_{\pm},\mathcal{L}_{\alpha}) \rightarrow H^{-\frac{1}{2}}(\Gamma,\mathbb{R}^{2}), \\
H^{1}(\mathfrak{D}_{\pm}, \mathcal{L}_{\alpha}) &\ni ({\bf u}, p; {\bf f}) \rightarrow \partial_{\nu;\alpha}({\bf u},p)_{{\bf f}} \in H^{-\frac{1}{2}}(\Gamma, \mathbb{R}^{2}),
\end{align}
defined by the formula \footnote{Let $ \langle \cdot, \cdot \rangle:= \ _{X'}\langle \cdot, \cdot \rangle_{X}$ denote the duality between the dual spaces $ X' $ and $ X $.}
\begin{align}
\label{conormal}
{\pm}\left\langle \partial_{\nu ;\alpha }^{\pm}({\bf u},p;{\bf f} ),
\varphi \right\rangle = 2&\langle {\mathbb E}({\bf
u}),{\mathbb E}(\rm{Tr}^{-1} \varphi )\rangle _{{\mathfrak{D}_{\pm}}}+\alpha \langle {\bf u},\rm{Tr}^{-1} \varphi \rangle _{{\mathfrak{D}_{\pm}}} \\ 
-&\langle p,{\rm{div}}\ \rm{Tr}^{-1} \varphi \rangle _{{\mathfrak{D}_{\pm}}} + \left\langle {\bf f},\rm{Tr}^{-1} \varphi \right\rangle _{{{\mathfrak{D}_{\pm}}}}, \forall \varphi \in H^{\frac{1}{2}}(\Gamma,\mathbb{R}^{2}) \nonumber
\end{align}
is a well-defined, linear and continuous operator, which does not dependent on the choice of the right
inverse $\rm{Tr}^{-1}$. Note that $ E_{jk}({\bf u}) $ are the components of the tensor field $ \mathbb{E}({\bf u}) = \frac{1}{2}(\nabla {\bf u} + (\nabla {\bf u})')$.
In addition the following Green formula holds
\begin{align}
\label{Green formula}
{\pm}\left\langle \partial_{\nu ;\alpha }^{\pm}({\bf u}, p ;{\bf f} ), \rm{Tr} \ {\bf u} \right\rangle \!=\! 2\langle {\mathbb E}({\bf u}),{\mathbb E}({\bf u} )\rangle _{{\mathfrak{D}_{\pm}}}\!+\! \alpha \langle {\bf u},{\bf u} \rangle _{{\mathfrak{D}_{\pm}}}  \!-\! \langle p ,{\rm{div}}\ {\bf u} \rangle _{{\mathfrak{D}_{\pm}}} \!+\! \left\langle {\bf f},{\bf u} \right\rangle _{{{\mathfrak{D}_{\pm}}}}, 
\end{align}
for all $ ({\bf u}, p ; {\bf f}) \in H^{1}(\mathfrak{D}_{\pm}, \mathcal{L}_{\alpha}) $.
\end{lemma}
Throughout this paper the notation $\partial^{+}_{\nu ;\alpha }(\textbf{u}, p) $ will denote the conormal derivative operator for the homogeneous problem, i.e., $\textbf{f} = \textbf{0}$. For simplicity, we will drop the superscript $+$ when working with a bounded Lipschitz domain $\mathfrak{D} = \mathfrak{D}_{+}$.

\subsection{The fundamental solution for the Brinkman operator in $\mathbb{R}^{n}$}
Let us denote by $\mathcal{G}^{\alpha} ({\bf x}, {\bf y}) $ and $\Pi^{\alpha} ({\bf x}, {\bf y}) $ \textit{the fundamental velocity tensor} and \textit{the fundamental pressure vector} for the Brinkman system in $\mathbb{R}^{2}$ given by
(see, e.g., \cite[(3.6)]{McCracken1981}, \cite[2.4.20]{KohrPop})
\begin{equation}
\label{E41}
\mathcal{G}^{\alpha}_{jk}({\bf x})=\frac{1}{2\pi} \left\{ \delta_{jk}  A_{1}(\sqrt{\alpha} |{\bf x}|)+ \frac{x_{j} x_{k}}{|{\bf x}|^{2}} A_{2}(\sqrt{\alpha} |{\bf x}|) \right\},\ \
\Pi_{k}^{\alpha}(x)= \frac{1}{2\pi}\frac{x_{k}}{|{\bf x}|^{2}}
\end{equation}
where $A_{1}(z)$ and $A_{2}(z)$ are defined by
\begin{align}
\label{A1A2}
A_{1}(z):=  K_{0}(z)+\frac{ K_{1}(z)}{z}-\dfrac{1}{z^{2}},\
A_{2}(z):= \frac{2}{z^{2}}- K_{2}(z),
\end{align}
$K_{{\varkappa} }$ is the Bessel function of the second kind and order ${\varkappa} \ge 0$.

\textit{The fundamental stress tensor} $\textbf{S}^{\alpha}$ and the \textit{the fundamental pressure tensor}$\boldsymbol\Lambda^{\alpha}$ have the components (see, e.g., \cite[(2.18)]{Varnhorn},  \cite[Section 2.3]{kohrpotanal1})
\begin{align}
\label{stress-tensor-alpha}
S^{\alpha }_{ij\ell }({\bf x}, {\bf y}) &=-{\Pi}^{\alpha}_{j}({\bf x}, {\bf y}) \delta_{i\ell } + \dfrac{\partial \mathcal{G}^{\alpha}_{ij}({\bf x}, {\bf y})}{\partial x_{\ell}} + \dfrac{\partial \mathcal{G}^{\alpha}_{\ell j}({\bf x}, {\bf y}) }{\partial x_{i}}, \quad \\
\label{dl-pressure2D}
\Lambda^{\alpha}_{ik}({\bf x},{\bf y}) &=\!\frac{1}{2\pi}\left\{-(y_i-x_i)\frac{4(y_k-x_k)}{|{\bf y}-{\bf x}|^{4}} \!-\! (\alpha |{\bf y}-{\bf x}|^{2} \log |{\bf y} -{\bf x}| \!+\! 2)\frac{\delta _{ik}}{|{\bf y}-{\bf x}|^2} \right\},
\end{align}
where $\delta_{jk} $ is the Kronecker symbol, ${\Pi}^{\alpha}_{j}({\bf x}, {\bf y})$ are the components of ${\Pi}^{\alpha}$, and $\mathcal{G}^{\alpha}_{ij}({\bf x}, {\bf y})$ are the components of $\mathcal{G}^{\alpha}({\bf x}, {\bf y})$.

For a given density $\textbf{g}\in H^{-\frac{1}{2}}(\Gamma, \mathbb{R}^{2})$,  the \textit{single-layer velocity potential} for the Brinkman system, ${\textbf{V}}_{\alpha; \Gamma }{\bf g}$, and the corresponding scalar \textit{pressure potential}, $Q^{s}_{\alpha; \Gamma }{\bf g}$, are given by
\begin{equation}
\label{single-layer}
(\textbf{V}_{\alpha; \Gamma } \textbf{g})(x) := \langle \mathcal{G}^{\alpha}(x,\cdot)|_{\Gamma}, {\bf g} \rangle_{\Gamma}, \
(Q^{s}_{\alpha; \Gamma } \textbf{g} )(x) := \langle {\Pi}^{\alpha}(x,\cdot)|_{\Gamma}, {\bf g}\rangle_{\Gamma}, \ x \in \mathbb{R}^{2}\setminus \Gamma.
\end{equation}

For a given density ${\textbf h}\in H^{\frac{1}{2}}(\Gamma, \mathbb{R}^{2})$ the \textit{double-layer velocity potential}, ${\bf W}_{\alpha; \Gamma  }{\bf h}$, and the corresponding scalar \textit{pressure potential}, $Q^{d}_{\alpha; \Gamma }\textbf{h}$, are defined by
\begin{align}
\label{dl-velocity}
&({\bf W}_{\alpha; \Gamma }{\textbf h})_{j}({\bf x}) := \int_{\Gamma} S^{\alpha}_{ij\ell } ({\bf x},{\bf y})\nu _{\ell}({\bf y}) h_{i}({\bf y}) d\sigma _y, \ \ \forall \ {\bf x}\in \mathbb{R}^{2}\setminus \Gamma,\\
\label{dl-pressure-field}
&(Q^{d}_{\alpha; \Gamma } \textbf{h})({\bf x}) := \int_{\Gamma} \Lambda^{\alpha }_{j\ell}({\bf x},{\bf y})\nu _{\ell}({\bf y}) h_{j}({\bf y}) d\sigma _y,\ \ \forall \ {\bf x}\in \mathbb{R}^{2}\setminus \Gamma,
\end{align}
where $\nu _\ell $, $\ell =1,2$, are the components of the outward unit normal $ \nu $ to ${\mathfrak{D}}$, which is defined almost everywhere (with respect to the surface measure $\sigma $) on $\Gamma$.

The main properties of layer potential {operators} for the Brinkman system are collected bellow (cf, e.g., \cite[Proposition 3.4]{Dindos}, \cite[Lemma 3.1]{kohrpotanal1}, \cite[Lemma 3.1]{K-L-W2} \cite[Theorem 3.1]{M-T} and \cite[Lemma 3.4]{K-L-M-W} for the layer potentials defined on $\mathbb{R}^{3}$).
\begin{theorem}
\label{layer-potential-properties3}
Let $\mathfrak{D}_{+}\subset {\mathbb R}^2$  be a bounded Lipschitz domain with connected boundary $\Gamma$ and let $\mathfrak{D}_{-}:={\mathbb R}^n\setminus \overline{\mathfrak{D}}_+$.
Let  $\alpha >0$.  Let ${\bf h}\in H^{\frac{1}{2}}(\Gamma,{\mathbb R}^2)$ and ${\bf g}\in H^{-\frac{1}{2}}(\Gamma,{\mathbb R}^2)$, then the following relations hold almost everywhere on $\Gamma $,
\begin{align}
\label{68}
&{\rm{Tr}^{+}} \big({\bf V}_{\alpha; \Gamma }{\bf g}\big)
={\rm{Tr}^{-}} \big({\bf V}_{\alpha; \Gamma  }{\bf g}\big)
=:{\mathcal V}_{\alpha; \Gamma  }{\bf g},
\quad\forall\ {\bf g}\in H^{-\frac{1}{2}}(\Gamma,{\mathbb R}^2);\\
\label{68-s1}
& \frac{1}{2}{\bf h}+ {\rm{Tr}}^{+}{({\bf W}_{\alpha; \Gamma}{\bf h})}
=-\frac{1}{2}{\bf h}+ {\rm{Tr}}^{-} {({\bf W}_{\alpha; \Gamma}{\bf h})}
=:{\bf K}_{\alpha; \Gamma  }{\bf h},
\quad\forall\ {\bf h}\in H^{\frac{1}{2}}(\Gamma,{\mathbb R}^2);\\
\label{70aaa}
-&\frac{1}{2}{\bf g}\!+\!\partial^{+}_{\alpha;\nu}\left({\bf V}_{\alpha; \Gamma  }{\bf g},{\mathcal Q}^{s}{\bf g}\right)
\!=\!\frac{1}{2}{\bf g}\!+\!\partial^{-}_{\alpha;\nu}\left({\bf V}_{\alpha; \Gamma  }{\bf g},{\mathcal Q}^{s}{\bf g}\right)
\!=:\!{\bf K}_{\alpha; \Gamma  }^*{\bf g},
\ \forall\ {\bf g}\in H^{-\frac{1}{2}}(\Gamma,{\mathbb R}^2);\\
\label{70aaaa}
&\partial^{+}_{\alpha;\nu}\big({\bf W}_{\alpha;\Gamma }{\bf h},{\mathcal Q}_{\alpha; \Gamma  }^d{\bf h}\big)
=\partial^{-}_{\alpha;\nu}\big({\bf W}_{\alpha;\Gamma }{\bf h},{\mathcal Q}_{\alpha; \Gamma  }^d{\bf h}\big)
=:-{\bf D}_{\alpha; \Gamma  }{\bf h},
\quad\forall\ {\bf h}\in H^{\frac{1}{2}}(\Gamma,{\mathbb R}^2);
\end{align}
where ${\bf K}_{\alpha; \Gamma}^*$ is the transpose of ${\bf K}_{\alpha;\Gamma}$ and the following integral operators are linear and bounded,
\begin{align}
\label{ss-s2}
&{\mathcal V}_{\alpha;\Gamma }:H^{-\frac{1}{2}}(\Gamma,{\mathbb R}^2)\to H^{\frac{1}{2}}(\Gamma,{\mathbb R}^2),\
{\bf K}_{\alpha;\Gamma }:H^{\frac{1}{2}}(\Gamma,{\mathbb R}^2)\to H^{\frac{1}{2}}(\Gamma,{\mathbb R}^2),\\
\label{ds-s2}
&{\bf K}_{\alpha;\Gamma }^*:H^{-\frac{1}{2}}(\Gamma,{\mathbb R}^2)\to H^{-\frac{1}{2}}(\Gamma,{\mathbb R}^2),\ {\bf D}_{\alpha;\Gamma}:H^{\frac{1}{2}}(\Gamma,{\mathbb R}^2)\to H^{-\frac{1}{2}}(\Gamma,{\mathbb R}^2).
\end{align}
\end{theorem}

\section{Formulation of the problem}
Let $ \mathfrak{D} \subset \mathbb{R}^{2}$ be a bounded Lipschitz domain with connected boundary $ \Gamma = \partial \mathfrak{D} $, which is decomposed into two adjacent, nonoverlapping parts $\Gamma_{D}, \Gamma_{N} $  as in definition \ref{Domain}.
%with the following properties
%\begin{equation}
%\label{Domain}
%\Gamma = \overline{\Gamma}_{D} \cup \overline{\Gamma}_{N}, \partial \Gamma_{D} = \partial \Gamma_{N} = \overline{\Gamma}_{D} \cap \overline{\Gamma}_{N}, \mbox{and} \ \ \mbox{meas} (\Gamma_{D}) > 0, \ \mbox{meas} (\Gamma_{N}) > 0.
%\end{equation}
%Note that, the positive measure of both partitions is critical to our case as will be explained in the sequel. 
Then, we consider the mixed problem with Dirichlet and Neumann boundary conditions for the Brinkman system
\begin{equation}
\label{Brinkman system}
\left\{
\begin{array}{l}
\Delta \textbf{u} - \alpha \textbf{u} - \nabla p = 0 \  \mbox{in} \ \mathfrak{D},  \\
\mbox{div} \ \textbf{u} = 0  \ \mbox{in} \ \mathfrak{D},  \\
(\mbox{Tr} \ \textbf{u})|_{\Gamma_D} = \textbf{f} \in H^{\frac{1}{2}}(\Gamma_{D},\mathbb{R}^{2}), \\
(\partial_{\nu;\alpha}(\textbf{u}, p ))_{\Gamma_{N}}  = \textbf{g} \in H^{-\frac{1}{2}}(\Gamma_{N},\mathbb{R}^{2}).
\end{array}
\right.
\end{equation}
where $(\cdot )|_{\Gamma_D}$ denotes the restriction operator from the Sobolev space $H^{\frac{1}{2}}(\Gamma,\mathbb{R}^{2})$ to $H^{\frac{1}{2}}(\Gamma_{D},\mathbb{R}^{2})$, and $(\cdot) |_{\Gamma_N}$ is the restriction from $H^{-\frac{1}{2}}(\Gamma,\mathbb{R}^{2})$ to $H^{-\frac{1}{2}}(\Gamma_{N},\mathbb{R}^{2})$.

In order to prove the well-posedness of the boundary value problem (\ref{Brinkman system}), we will reformulate the problem as a system of boundary integral equations (BIE's), inspired by the main ideas in \cite{K-W}. We start with the Green representation formula of a weak solution (see, e.g., \cite[Lemma 3.7]{GuttMMA} and also \cite[(2.3.13) and (2.3.14)]{H-W} and \cite[Proposition 4.15]{M-W} for $\alpha = 0$)
\begin{equation}
\label{direct Approach}
\textbf{u}(x) = \textbf{V}_{\alpha;\Gamma} (\partial_{\nu;\alpha}(\textbf{u}, p )) - \textbf{W}_{\alpha; \Gamma} (\mbox{Tr} \ \textbf{u}), \quad
p(x) = Q^{s}_{\alpha;\Gamma}  (\partial_{\nu;\alpha}(\textbf{u}, p)) - Q^{d}_{\alpha; \Gamma}  ({\rm Tr} \ \textbf{u}).
\end{equation}

Letting $x \to \Gamma$ from inside of $\mathfrak{D}$, and following the jump relations for a double-layer potential and the conormal derivation of a single-layer potential, we obtain the following integral equations on $\Gamma$ for the Brinkman system (\ref{Brinkman system})
\begin{align}
\label{CP1}
 \mathcal{V}_{\alpha;\Gamma}(\partial_{\nu;\alpha}(\textbf{u},p)) -  \left( \frac{1}{2} \mathbb{I} + \textbf{K}_{\alpha;\Gamma}  \right){\rm Tr} \ \textbf{u}  = 0  \\
\label{CP2}
\left(- \frac{1}{2} \mathbb{I} + \textbf{K}^{*}_{\alpha;\Gamma}  \right)(\partial_{\nu;\alpha}(\textbf{u}, p)) + \textbf{D}_{\alpha;\Gamma} {\rm Tr} \  \textbf{u} = 0.
\end{align}

%In order to match the system composed of \eqref{CP1} and \eqref{CP2} to the mixed boundary value problem (\ref{Brinkman system}), let us denote by $\tilde{\textbf{f}} \in H^{\frac{1}{2}}(\Gamma, \mathbb{R}^{2})$, $\tilde{\textbf{g}} \in  H^{-\frac{1}{2}}(\Gamma, \mathbb{R}^{2})$,  arbitrarily chosen but fixed extensions to the entire $\Gamma$ of the given
%functions  $\textbf{f}$, $\textbf{g}$, 
According to the definition of the spaces $H^{\frac{1}{2}}(\Gamma_{D}, \mathbb{R}^{2})$ and $ H^{-\frac{1}{2}}(\Gamma_{N}, \mathbb{R}^{2})$, there exist some extensions $\tilde{\bf f} \in H^{\frac{1}{2}}(\Gamma, \mathbb{R}^{2})$ and $\tilde{\bf g} \in H^{-\frac{1}{2}}(\Gamma, \mathbb{R}^{2})$ such that $\tilde{\bf f}|_{\Gamma_{D}} = {\bf f}$ and $\tilde{\bf g}|_{\Gamma_{N}} = {\bf g}$. Moreover, we assume without loss of generality that the extension $\tilde{\bf f}$ of the Dirichlet data ${\bf f}$ satisfies the orthogonality relation $\langle \tilde{\bf f}, \nu \rangle = 0$, i.e., $\tilde{\bf f} \in H^{\frac{1}{2}}_{\nu}(\Gamma,\mathbb{R}^{2})$ (see, e.g., \cite[3.45]{K-W}).

For the sake of completion, we show that for every ${\bf f} \in H^{\frac{1}{2}}(\Gamma_{D},\mathbb{R}^{2})$, there exists an extension $\tilde{\bf f} \in  H^{\frac{1}{2}}_{\nu}(\Gamma,\mathbb{R}^{2})$. To this end, we notice that $
\nu|_{\Gamma_{N}} \in H^{-\frac{1}{2}}(\Gamma_{N},\mathbb{R}^{2}) = (\widetilde{H}^{\frac{1}{2}}(\Gamma_{N},\mathbb{R}^{2}))'$.
Since $ \widetilde{H}^{\frac{1}{2}}(\Gamma_{N},\mathbb{R}^{2})$ is a Hilbert space, the Riesz representation theorem states that there exists a unique $\mu \in \widetilde{H}^{\frac{1}{2}}(\Gamma_{N},\mathbb{R}^{2}) \subset H^{\frac{1}{2}}(\Gamma,\mathbb{R}^{2})$ which satisfies the equation
\begin{equation}
\label{MuDef}
\langle \nu|_{\Gamma_{N}}, \phi \rangle_{\Gamma_{N}} = \langle \mu, \phi \rangle_{\widetilde{H}^{\frac{1}{2}}(\Gamma_{N}.\mathbb{R}^{2})},
\end{equation}
For the sake of brevity, we denote by $\hat{\mu} := \frac{1}{\| \mu \|^{2}} \mu$ the normed function. Therefore by \eqref{MuDef}, we get that $\langle \nu|_{\Gamma_{N}}, \hat{\mu} \rangle_{\Gamma_{N}} = 1$.

Then, for a fixed extension $\tilde{\bf F} \in \widetilde{H}^{\frac{1}{2}}(\Gamma, \mathbb{R}^{2})$ of the Dirichlet data ${\bf f}$, we construct $\tilde{\bf f}$ in the following form
\begin{equation}
\tilde{\bf f} = \tilde{\bf F} - \langle \tilde{\bf F}, \nu \rangle \hat{\mu}.
\end{equation}
Clearly $\tilde{\bf f} \in H^{\frac{1}{2}}(\Gamma,\mathbb{R}^{2})$ since $\hat{\mu} \in \widetilde{H}^{\frac{1}{2}}(\Gamma_{N},\mathbb{R}^{2}) \subset H^{\frac{1}{2}}(\Gamma,\mathbb{R}^{2})$. Taking the restriction to $\Gamma_{D}$, we obtain $
\tilde{\bf f}|_{\Gamma_{D}} = \tilde{\bf F}|_{\Gamma_{D}} - \langle \tilde{\bf F}, \nu \rangle_{_{\Gamma}} \hat{\mu}|_{\Gamma_{D}} = {\bf f}$, since $\hat{\mu}$ has support on $\overline{\Gamma_{N}}$. Finally, the following computation based on the norm of $\hat{\mu}$ yields
\begin{equation}
\langle \tilde{\bf f}, \nu \rangle = \langle \tilde{\bf F}, \nu \rangle - \langle \tilde{\bf F}, \nu\rangle \langle \hat{\mu}, \nu \rangle = 0.
\end{equation}
and implies that $\tilde{\bf f} \in  \widetilde{H}^{\frac{1}{2}}_{\nu}(\Gamma,\mathbb{R}^{2})$.

By the above assumptions the boundary data for the Brinkman system \eqref{Brinkman system} is given by
\begin{equation}
\label{restriction}
\left( {\rm Tr} \ \textbf{u}\right)|_{\Gamma} = \varphi_{N} + \tilde{\bf f}, \quad \left( \partial_{\nu;\alpha} (\textbf{u}, p) \right)|_{\Gamma} = \psi_{D} + \tilde{\textbf{g}}.
\end{equation}
for some $\varphi_{N} \in \widetilde{H}^{\frac{1}{2}}(\Gamma_{N}, \mathbb{R}^{2})$ and $\psi_{D} \in \widetilde{H}^{-\frac{1}{2}}(\Gamma_{D}, \mathbb{R}^{2})$. The condition $\tilde{\bf f} \in  \widetilde{H}^{\frac{1}{2}}_{\nu}(\Gamma,\mathbb{R}^{2})$, the continuity equation and the flux-divergence theorem, imply that the desired density $\varphi_{N}$ satisfies also the orthogonality condition $\langle \varphi_{N}, \nu \rangle = 0$, i.e., $\varphi \in \widetilde{H}^{\frac{1}{2}}_{\nu}(\Gamma, \mathbb{R}^{2})$.

By restricting (\ref{CP1}) to $\Gamma_{D}$ and (\ref{CP2}) to $\Gamma_{N}$ we obtain
the following system of boundary integral equations with the unknowns $\psi_{D} \in \widetilde{H}^{-\frac{1}{2}}(\Gamma_{D}, \mathbb{R}^{2})$ and $\varphi_{N} \in \widetilde{H}_{\nu}^{\frac{1}{2}}(\Gamma_{N}, \mathbb{R}^{2})$
\begin{equation}
\left\{
\begin{array}{l}
\label{Eq1}
\mathcal{V}_{\alpha;\Gamma} \psi_{D} - \textbf{K}_{\alpha;\Gamma} \varphi _{N}  =  f_{1}, \quad x \in \Gamma_{D} \\
\textbf{K}^{*}_{\alpha;\Gamma} \psi _{D}  + \textbf{D}_{\alpha; \Gamma}\varphi_{N}  =  f_{2}, \quad x \in \Gamma_{N}
\end{array}
\right.
\end{equation}
where
\begin{align}
f_{1} =  \frac{1}{2}\tilde{\textbf{f}} + \textbf{K}_{\alpha;\Gamma}\tilde{\textbf{f}} - \mathcal{V}_{\alpha;\Gamma} \tilde{\textbf{g}},  \quad 
f_{2} =   -\textbf{D}_{\alpha;\Gamma} \tilde{\textbf{f}}  + \frac{1}{2}\tilde{\textbf{f}} - \textbf{K}^{*}_{\alpha;\Gamma} \tilde{\textbf{g}}.
\end{align}

\subsection{Variational Formulation for the system of boundary integral equations \textbf{\eqref{Eq1}}}

For simplicity, we introduce the notation $ \mathcal{H} $ for the product space
\begin{equation}
\mathcal{H} := \widetilde{H}^{\frac{1}{2}}(\Gamma_{N}, \mathbb{R}^{2}) \times \widetilde{H}^{-\frac{1}{2}}(\Gamma_{D}, \mathbb{R}^{2}) \subset \widetilde{H}^{\frac{1}{2}}_{\nu}(\Gamma, \mathbb{R}^{2}) \times \widetilde{H}^{-\frac{1}{2}}(\Gamma, \mathbb{R}^{2}).
\end{equation}
Let us define the bilinear form $ a: \mathcal{H} \times \mathcal{H} \to \mathbb{R} $ as
\begin{equation}
\label{bilinear form}
a((\varphi_{N},\psi_{D});(\varphi,\psi)) := \langle \mathcal{V}_{\alpha;\Gamma} \psi_{D}, \psi \rangle - \langle \textbf{K}_{\alpha;\Gamma} \varphi_{N},\psi \rangle + \langle \textbf{K}^{*}_{\alpha;\Gamma} \psi_{D}, \varphi \rangle + \langle \textbf{D}_{\alpha; \Gamma} \varphi_{N}, \varphi \rangle. 
\end{equation}

Then the equivalent variational formulation for the mixed Dirichlet-Neumann boundary value problem for the Brinkman system (\ref{Brinkman system}) is to find $ (\varphi_{N},\psi_{D}) \in \mathcal{H} $ such that the following equation is satisfied (see \cite{K-W} for $ \alpha = 0$ )
\begin{equation}
\label{Var Brinkman system}
a((\varphi_{N},\psi_{D}); (\varphi,\psi)) = l(\varphi,\psi), \qquad \forall
(\varphi, \psi) \in \mathcal{H} 
\end{equation}
where 
\begin{equation}
l(\varphi,\psi) = \langle \frac{1}{2}\tilde{\bf f} + \textbf{K}_{\alpha;\Gamma} \tilde{\textbf{f}} - \mathcal{V}_{\alpha;\Gamma} \tilde{\textbf{g}}, \psi \rangle + \langle -\textbf{D}_{\alpha;\Gamma} \tilde{\textbf{f}}  + \frac{1}{2}\tilde{\bf g} - \textbf{K}^{*}_{\alpha;\Gamma}\tilde{\textbf{g}}, \varphi\rangle.
\end{equation}

In order to analyse the variational problem (\ref{Var Brinkman system}), we study the coerciveness properties of the hypersingular boundary integral operator $\textbf{D}_{\alpha;\Gamma}$ and the single-layer integral operator  $\mathcal{V}_{\alpha;\Gamma}$. To this purpose, we show the following Garding inequalities (see \cite[Lemma 3.9]{K-W} for the Stokes system).

\begin{lemma}
\label{Lemma3} Let $ \mathfrak{D} \subset \mathbb{R}^{2}$ be a bounded Lipschitz domain with connected boundary $ \Gamma $. There exists a compact operator  $\mathcal{C}_{D} : H^{\frac{1}{2}}(\Gamma, \mathbb{R}^{2}) \to $ $ H^{-\frac{1}{2}}(\Gamma, \mathbb{R}^{2})$, such that 
\begin{equation}
\label{EqLemma3}
\langle (\textbf{D}_{\alpha;\Gamma} + \mathcal{C}_{D}) \varphi, \varphi \rangle \geq c_{1} \|\varphi\|^{2}_{H^{\frac{1}{2}}(\Gamma, \mathbb{R}^{2})}, \ \forall \varphi \in H^{\frac{1}{2}}(\Gamma, \mathbb{R}^{2}).
\end{equation}
\end{lemma}
\begin{proof}
We follow the main ideas in the proof of \cite[Lemma 3.9]{K-W}.
For any $\varphi \in  H^{\frac{1}{2}}(\Gamma, \mathbb{R}^{2})$, let us consider the double-layer potentials
\begin{equation}
\textbf{u}_{0} = \textbf{W}_{\alpha; \Gamma}\varphi, \quad p_{0} = Q^{d}_{\alpha;\Gamma}\varphi  \qquad x \in \mathbb{R}^{2}\backslash \Gamma.
\end{equation}
Then $(\textbf{u}_{0},p_{0}) \in H^{1}(\mathfrak{D}_{\pm}, \mathbb{R}^{2}) \times L^{2}(\mathfrak{D}_{\pm})$. Moreover, we have the jump relations for the double-layer potential operator (see, e.g., \cite[Theorem 2.1]{Dindos}, \cite[Lemma 3.1]{K-L-W2})
\begin{align}
\label{jump3}
&{\rm{Tr}}^{-}\big(\textbf{W}_{\alpha ,\Gamma}{\varphi}\big) - {\rm{Tr}}^{+}\big(\textbf{W}_{\alpha ,\Gamma}{\varphi}\big) = \varphi, \\
\label{jump4}
&{\partial }_{\nu ;\alpha }^{-}\left(\textbf{W}_{\alpha ;\Gamma}{\varphi},{\mathcal Q}_{\alpha ;\Gamma}^s{\varphi}\right) - \partial_{\nu ;\alpha }^{+}\left(\textbf{W}_{\alpha ,\Gamma }{\varphi},{\mathcal Q}_{\alpha ,\Gamma}^s{\varphi}\right) = - \textbf{D}_{\alpha;\Gamma}\varphi + \textbf{D}_{\alpha;\Gamma}\varphi = 0.
\end{align}
Considering a constant $\eta> 0$, such that the domain $\mathfrak{D}$ lies in the closed ball $B_{\eta}:= \{ x \in \mathbb{R}^{2}: |x| \leq \eta\}$, i.e., $\mathfrak{D} \subset B_{\eta}$, we obtain the following identities based on the Green formula (\ref{Green formula}) for the interior domain respectively for the intersection of the exterior domain with the closed ball $ \mathfrak{D}_{\eta} :=  \mathfrak{D}_{-} \cap B_{\eta} $, i.e.,
\begin{align}
\label{intD_GF}
\langle \partial^{+}_{\nu;\alpha}(\textbf{u}_{0},p_{0}),\rm{Tr}^{+} \ \textbf{u}_{0} \rangle  &=  2 \langle \mathbb{E}(\textbf{u}_{0}), \mathbb{E}(\textbf{u}_{0}) \rangle_{\mathfrak{D}_{+}} + \alpha \langle \textbf{u}_{0}, \textbf{u}_{0} \rangle_{\mathfrak{D}_{+}},  \\
\label{extD-GF}
-\langle \partial^{-}_{\nu;\alpha}(\textbf{u}_{0}, p_{0}),\rm{Tr}^{-} \ \textbf{u}_{0} \rangle &+  \langle \partial^{+}_{\nu;\alpha}(\textbf{u}_{0}, p_{0}),\rm{Tr}^{+} \ \textbf{u}_{0} \rangle_{\partial B_{\eta}}  =  2 \langle \mathbb{E}(\textbf{u}_{0}), \mathbb{E}(\textbf{u}_{0}) \rangle_{\mathfrak{D}_{\eta}} \nonumber \\
&+ \alpha \langle \textbf{u}_{0}, \textbf{u}_{0} \rangle_{\mathfrak{D}_{\eta}} . 
\end{align}

Adding equations \eqref{intD_GF} and \eqref{extD-GF} and applying Korn's inequality (see, e.g., \cite[Lemma 5.4.4]{H-W}) to the right hand side, we obtain
\begin{align}
\label{KornIneq}
\! \langle \textbf{D}_{\alpha;\Gamma}\varphi, \varphi \rangle \! + \! \langle \partial^{+}_{\nu;\alpha}(\textbf{u}_{0}, p_{0}),\rm{Tr}^{+} \ \textbf{u}_{0} \rangle_{\partial B_{\eta}} \! \geq \ &c_{0}(1+\alpha) \|\textbf{u}_{0}\|^{2}_{H^{1}(\mathfrak{D}, \mathbb{R}^{2})} - \|\textbf{u}_{0}\|^{2}_{L^{2}(\mathfrak{D}, \mathbb{R}^{2})} \nonumber \\
+  &c_{0}(1+\alpha)\| \textbf{u}_{0}\|^{2}_{H^{1}(\mathfrak{D}_{\eta},\mathbb{R}^{2})} \!\! - \! \| \textbf{u}_{0}\|^{2}_{L^{2}(\mathfrak{D}_{\eta}, \mathbb{R}^{2})}. 
\end{align}
Let us consider the linear operator defined by 
\begin{align}
\label{cmp3}
\!\!\! \langle \mathcal{C}_{D} u,\varphi  \rangle \! = \!\langle u,\textbf{W}^{*}_{\alpha,\mathfrak{D}_{+}} \textbf{W}_{\alpha,\mathfrak{D}_{+}}\varphi \rangle \! &+\! \langle u ,\textbf{W}^{*}_{\alpha,\mathfrak{D}_{-}} \textbf{W}_{\alpha,\mathfrak{D}_{-}}\varphi \rangle  \! \nonumber \\
&+ \!\langle \partial^{+}_{\nu;\alpha}(\textbf{W}_{\alpha,\mathfrak{D}_{-}}u, Q^{d}_{\alpha,\mathfrak{D}_{-}}u),\rm{Tr}^{+} \ \textbf{W}_{\alpha,\mathfrak{D}_{-}}\varphi \rangle_{\partial B_{\eta}}. 
\end{align}
for every $\varphi \in H^{\frac{1}{2}}(\Gamma,\mathbb{R}^{2}) $. Note that the operator \eqref{cmp3} is well-defined, due to the embedding $ H^{1}(\mathfrak{D}_{+},\mathbb{R}^{2}) \hookrightarrow L^{2}(\mathfrak{D}_{+},\mathbb{R}^{2})$, by which we can define the adjoint double-layer potential $\textbf{W}^{*}_{\alpha,\mathfrak{D}_{+}} : L^{2}(\mathfrak{D}_{+},\mathbb{R}^{2}) \to H^{\frac{1}{2}}(\Gamma,\mathbb{R}^{2})$  as follows
\begin{equation}
\label{cmp33}
\langle \textbf{W}_{\alpha;\mathfrak{D}_{+}} u, \textbf{W}_{\alpha;\mathfrak{D}_{+}} \varphi \rangle = \langle u, \textbf{W}^{*}_{\alpha,\mathfrak{D}_{+}} \textbf{W}_{\alpha;\mathfrak{D}_{+}} \varphi \rangle,
\end{equation}
where the notation $\langle \cdot, \cdot \rangle $ in the left hand side refers to the inner product on the space $L^{2}(\mathfrak{D}_{+},\mathbb{R}^{2})$, while the same notation in the right hand side means the inner product of the space $H^{\frac{1}{2}}(\mathfrak{D}_{+}, \mathbb{R}^{2})$. However, for the same of brevity, we use the same notation in both sides of \eqref{cmp33}. According to the compact embedding $H^{1}(\mathfrak{D}_{+},\mathbb{R}^{2})$ into $L^{2}(\mathfrak{D}_{+},\mathbb{R}^{2})$ and for a well chosen constant $\eta$ such that $\mbox{dist} \ (\Gamma, \partial B_{\eta}) > 0$, we deduce that the operator $\mathcal{C}_{D}$ is compact.

Finally from \eqref{KornIneq} and the continuity of the trace operator, we obtain the following relation
\begin{align}
\label{cmp1}
\!\!\langle (\textbf{D}_{\alpha;\Gamma} +\mathcal{C}_{D}) \varphi, \varphi \rangle \! &\geq c_{1} \Big \{ \| \Big(-\frac{1}{2}\mathbb{I} \! + \! \textbf{K}_{\alpha;\Gamma} \Big) \varphi \|^{2}_{H^{\frac{1}{2}}(\Gamma,\mathbb{R}^{2})} \!\!\! + \|\Big(\frac{1}{2}\mathbb{I} \! + \! \textbf{K}_{\alpha;\Gamma} \Big) \varphi \|^{2}_{H^{\frac{1}{2}}(\Gamma,\mathbb{R}^{2})}\Big\} \nonumber \\
&\geq c_{1} \| \varphi \|^{2}_{H^{\frac{1}{2}}(\Gamma,\mathbb{R}^{2})},
\end{align}
which proves our assertion.

\end{proof}

\begin{lemma}
\label{Lemma2}  Let $ \mathfrak{D} \subset \mathbb{R}^{2}$ be a bounded Lipschitz domain with connected boundary $ \Gamma $. Then there exists a compact operator  $\mathcal{C}_{V} : H^{-\frac{1}{2}}(\Gamma, \mathbb{R}^{2}) \to $ $ H^{\frac{1}{2}}(\Gamma, \mathbb{R}^{2})$ such that 
\begin{equation}
\langle (\mathcal{V}_{\alpha;\Gamma} + C_{V}) \psi, \psi \rangle \geq c_{2} \|\psi\|^{2}_{H^{-\frac{1}{2}}(\Gamma, \mathbb{R}^{2})}, \quad \forall \psi \in H^{-\frac{1}{2}}(\Gamma, \mathbb{R}^{2}).
\end{equation}
\end{lemma}
\begin{proof}
For any $\psi \in  H^{-\frac{1}{2}}(\Gamma, \mathbb{R}^{2})$, let us consider the single-layer potential operators
\begin{equation}
\textbf{u}_{0} = \textbf{V}_{\alpha; \Gamma}\psi, \ p_{0} = Q^{s}_{\alpha;\Gamma}\psi, \qquad x \in \mathbb{R}^{2}\backslash \Gamma.
\end{equation}
Then $(\textbf{u}_{0},p_{0}) \in H^{1}(\mathfrak{D}_{\pm}, \mathbb{R}^{2}) \times L^{2}(\mathfrak{D}_{\pm})$. The jump relations for this particular fields yield
\begin{equation}
\label{jump1}
{\rm{Tr}}^{+}\big({\textbf{V}}_{\alpha; \Gamma}{\psi}\big) - {\rm{Tr}}^{-}\big({\textbf{V}}_{\alpha; \Gamma}{\psi}\big) = 0,
\end{equation}
\begin{equation}
\label{jump2}
{\partial }_{\nu ;\alpha }^{+}\left({\textbf{V}}_{\alpha ;\Gamma}{\psi},{\mathcal Q}_{\alpha ;\Gamma}^s{\psi}\right) - \partial_{\nu ;\alpha }^{-}\left({\textbf{V}}_{\alpha ;\Gamma }{\psi},{\mathcal Q}_{\alpha ;\Gamma}^s{\psi}\right) = \psi.
\end{equation}
As in Lemma \ref{Lemma3}, we consider again a constant $\eta$ and the corresponding closed ball $ B_{\eta} := \{ x \in \mathbb{R}^{2}: |x| \leq \eta\}$, such that $\mathfrak{D} \subset B_{\eta}$ and let $\mathfrak{D}_{\eta}$ be the intersection of the exterior domain with the closed ball, i.e., $ \mathfrak{D}_{\eta} :=  \mathfrak{D}_{-} \cap B_{\eta} $. By the jump relations (\ref{jump1}) and (\ref{jump2}) and the Green formula (\ref{Green formula}), we can write
\begin{align}
\!\!\! \langle \mathcal{V}_{\alpha; \Gamma} \psi ,\psi \rangle + \langle \partial^{+}_{\nu;\alpha}(\textbf{u}_{0}, p_{0}),\rm{Tr}^{+} \ \textbf{u}_{0} \rangle_{\partial B_{\eta}}  \!&= \! 2 \langle \mathbb{E}( \textbf{u}_{0}) ,\mathbb{E}( \textbf{u}_{0}) \rangle_{\mathfrak{D}_{+} }\! + \! 2 \langle \mathbb{E}( \textbf{u}_{0}) ,\mathbb{E}( \textbf{u}_{0}) \rangle_{ \mathfrak{D}_{\eta} } \nonumber \\
&+ \alpha \langle \textbf{u}_{0},  \textbf{u}_{0} \rangle_{\mathfrak{D}_{+}} + \alpha \langle \textbf{u}_{0},  \textbf{u}_{0} \rangle_{\mathfrak{D}_{\eta}}. \nonumber
\end{align}
Then the result follows with similar arguments as for Lemma \ref{Lemma3}, where the compact operator is given by
\begin{align}
\label{cmpV}
\!\!\! \langle \mathcal{C}_{V} u,\psi  \rangle &= \langle u, \textbf{V}^{*}_{\alpha;\mathfrak{D}_{+}}\textbf{V}_{\alpha;\mathfrak{D}_{+}}\psi \rangle + \langle u, \textbf{V}^{*}_{\alpha;\mathfrak{D}_{-}}\textbf{V}_{\alpha;\mathfrak{D}_{-}}\psi \rangle  \nonumber \\
&+\langle \partial^{+}_{\nu;\alpha}(\textbf{V}_{\alpha,\mathfrak{D}_{-}}u, Q^{s}_{\alpha,\mathfrak{D}_{-}}u),\rm{Tr}^{+} \ \textbf{V}_{\alpha,\mathfrak{D}_{-}}\varphi \rangle_{\partial B_{\eta}}. 
\end{align}
\end{proof}

Moreover, we need the following positivity lemma for the single-layer and the hypersingular potential operators (see, e.g., \cite[Theorem 3.8]{K-W}).
\begin{lemma}
\label{positiveD}
Let $ \mathfrak{D} \subset \mathbb{R}^{2}$ be a bounded Lipschitz domain with connected boundary $ \Gamma $ as in definition \ref{Domain}. There exists two positive constants $c_{V}>0$  and $c_{D}>0$, such that 
\begin{equation}
\label{PositivV}
\langle \mathcal{V}_{\alpha;\Gamma} \psi, \psi \rangle  \geq c_{V} \|\psi \|^{2}_{\widetilde{H}^{-\frac{1}{2}}(\Gamma_{D},\mathbb{R}^{2})}, \quad \forall \psi \in \widetilde{H}^{-\frac{1}{2}}(\Gamma_{D},\mathbb{R}^{2})
\end{equation}
and
\begin{equation}
\label{PositivD}
\langle \textbf{D}_{\alpha;\Gamma} \varphi, \varphi \rangle  \geq c_{D} \| \varphi \|^{2}_{H^{\frac{1}{2}}_{\nu}(\Gamma,\mathbb{R}^{2})}, \quad \forall \varphi \in H^{\frac{1}{2}}_{\nu}(\Gamma,\mathbb{R}^{2}).
\end{equation}
\end{lemma}

\begin{proof}
Let $(\textbf{u}_{0}, p_{0}) := (\textbf{W}_{\alpha; \Gamma}\varphi, Q^{d}_{\alpha;\Gamma} \varphi)$ be the double-layer velocity and pressure potentials with a density $\varphi \in H^{\frac{1}{2}}_{\nu}(\Gamma, \mathbb{R}^{2})$.
Note that for $\varphi \in H^{\frac{1}{2}}_{\nu}(\Gamma,\mathbb{R}^{2})$, the asymptotic behavior of the double-layer velocity operator $\textbf{W}_{\alpha;\Gamma}$, the gradient of the double-layer potential operator $\nabla \textbf{W}_{\alpha}$ and the double-layer pressure potential operator $Q^{d}_{\alpha}$ are the following
(see, e.g., \cite[3.12]{kohrpotanal1}, \cite[Lemma 3.7.3]{KohrPop}, \cite[Lemma 2.12]{Varnhorn}) 
\begin{equation}
\label{Asymptotic1}
\left(\textbf{W}_{\alpha} \varphi \right)({\bf x}) = \mathcal{O}(|{\bf x}|^{-2}), \ (\nabla \textbf{W}_{\alpha} \varphi)({\bf x}) = \mathcal{O}(|{\bf x}|^{-1}), \ (Q^{d}_{\alpha} \varphi)({\bf x}) =  \mathcal{O}(|{\bf x}|^{-1}).
\end{equation}

Hence, for $\varphi \in H^{\frac{1}{2}}_{\nu}(\Gamma,\mathbb{R}^{2})$, we deduce that there exists a constant $M_{1}>0$ such that
\begin{align}
\label{ballInfinity2}
\langle {\bf t}^{+}_{\alpha}(\textbf{u}_{0}, \pi_{0}),\gamma^{+} \textbf{u}_{0} \rangle_{\partial B_{\eta}} &\leq \int_{\partial B_{\eta}} |  {\bf t}^{+}_{\alpha}(\textbf{u}_{0}, \pi_{0})||\gamma^{+} \textbf{u}_{0}| d\sigma_{\eta} \nonumber \\
&= \frac{1}{\eta^{3}} \int_{0}^{2\pi} \eta |  {\bf t}^{+}_{\alpha}(\textbf{u}_{0}, \pi_{0})| \eta^{2}|\gamma^{+} \ \textbf{u}_{0}| \eta d\theta \nonumber  \\
&\leq  \frac{1}{\eta^{2}} M_{1} \to 0, \quad \mbox{as } |{\bf x}| \to \infty.
\end{align}

This implies, by adding \eqref{intD_GF} and \eqref{extD-GF} that
\begin{equation}
\langle \textbf{D}_{\alpha;\Gamma} \varphi, \varphi \rangle  = 2 \langle \mathbb{E}( \textbf{u}_{0}) ,\mathbb{E}(\textbf{u}_{0}) \rangle_{\mathfrak{D}_{+}} + 2 \langle \mathbb{E}( \textbf{u}_{0}) ,\mathbb{E}(\textbf{u}_{0}) \rangle_{\mathfrak{D}_{-}} + \alpha \langle  \textbf{u}_{0}, \textbf{u}_{0} \rangle_{\mathfrak{D}_{+}}  + \alpha \langle  \textbf{u}_{0}, \textbf{u}_{0} \rangle_{\mathfrak{D}_{-}},
\end{equation}
which means that $\textbf{D}_{\alpha;\Gamma} $ is non-negative and we show that the hypersingular potential operator is positive for $0 \ne \varphi \in H^{\frac{1}{2}}(\Gamma, \mathbb{R}_{2})$.
To this purpose, let us assume that 
\begin{equation}
\langle \textbf{D}_{\alpha;\Gamma} \varphi, \varphi \rangle  = 0
\end{equation}
which gives the relations $\langle \mathbb{E}(\textbf{u}_{0}) ,\mathbb{E}(\textbf{u}_{0}) \rangle_{\mathfrak{D}_{\pm}} = 0 \ \mbox{and } \ \langle \textbf{u}_{0}, \textbf{u}_{0} \rangle_{\mathfrak{D}_{\pm}} = 0$.

Therefore, $\textbf{u}_{0} = 0 $ in $\mathfrak{D}_{\pm}$ and the jump relation $\rm{Tr}^{-}\textbf{u}_{0} - \rm{Tr}^{+}\textbf{u}_{0} = \varphi$, implies $\varphi = 0$. Consequently, the relation (\ref{PositivD}) holds in addition to Lemma \ref{Lemma3}, which implies that the operator $\textbf{D}_{\alpha;\Gamma}$ is $H^{\frac{1}{2}}(\Gamma,\mathbb{R}^{2})$-elliptic (cf., e.g, \cite[$\S$6.2 and $\S$6.11]{T}, \cite[Lemma 5.2.5]{H-W}; see also \cite[Theorem 3.8]{K-W} in the case of the Stokes system). 

Now let us show the estimate \eqref{PositivD}. To this end, let $(\textbf{u}_{0}, p_{0}) := (\textbf{V}_{\alpha; \Gamma}\psi, Q^{s}_{\alpha;\Gamma}\psi )$ be the single-layer potentials with the density $\psi \in H^{-\frac{1}{2}}(\Gamma,\mathbb{R}^{2})$. The behavior at infinity of the single-layer potential are the following (see, e.g., \cite[3.12]{kohrpotanal1},  \cite[Lemma 3.7.3]{KohrPop}, \cite[Lemma 2.12]{Varnhorn})
\begin{equation}
\left( \textbf{V}_{\alpha;\Gamma} \psi\right)({\bf x}) = \mathcal{O}(|{\bf x}|^{-2}), \ \ \left( \nabla\textbf{V}_{\alpha;\Gamma} \psi\right)({\bf x}) = \mathcal{O}(|{\bf x}|^{-1}), \ \ \left( Q^{s}_{\alpha;\Gamma} \psi\right)({\bf x}) = \mathcal{O}(|{\bf x}|^{-1}),
\end{equation}
which imply by similar arguments as in \eqref{ballInfinity2} that  $ \langle \partial^{+}_{\nu;\alpha}(\textbf{u}_{0}, p_{0}),\rm{Tr}^{+} \ \textbf{u}_{0} \rangle_{\partial B_{\eta}} \to 0$ as $\eta \to \infty$. Hence, we obtain
\begin{equation}
\label{VDD}
\langle \mathcal{V}_{\alpha;\Gamma} \psi, \psi \rangle \! =\! 2 \langle \mathbb{E}( \textbf{u}_{0}) ,\mathbb{E}(\textbf{u}_{0}) \rangle_{\mathfrak{D}_{+}} + 2 \langle \mathbb{E}( \textbf{u}_{0}) ,\mathbb{E}(\textbf{u}_{0}) \rangle_{\mathfrak{D}_{-}} + \alpha \langle  \textbf{u}_{0}, \textbf{u}_{0} \rangle_{\mathfrak{D}_{+}}  + \alpha \langle  \textbf{u}_{0}, \textbf{u}_{0} \rangle_{\mathfrak{D}_{-}}.
\end{equation}
The relation \eqref{VDD} implies that $\langle \mathcal{V}_{\alpha;\Gamma} \psi, \psi \rangle \geq 0$ for all $\psi \in H^{-\frac{1}{2}}(\Gamma,\mathbb{R}^{2})$, in particular $\psi \in \widetilde{H}^{-\frac{1}{2}}(\Gamma_{D},\mathbb{R}^{2})$. Let now $\psi \in \widetilde{H}^{-\frac{1}{2}}(\Gamma_{D},\mathbb{R}^{2})$, such that $\langle \mathcal{V}_{\alpha;\Gamma} \psi, \psi \rangle = 0$. Then ${\bf u}_{0} = 0$ on $\mathfrak{D}_{+} \cup \mathfrak{D}_{-}$ by \eqref{VDD}. Since $({\bf u}_{0},p_{0})$ is a solution of the Brinkman system, there exists two constants $c_{+}, c_{-}$, such that $p_{ 0} = c_{\pm}$ on $\mathfrak{D}_{\pm}$. Since $\psi = 0$ on $\Gamma_{N}$ and meas($\Gamma_{N}$) $>$ 0, we infer that $(c_{-} - c_{+}) = 0$. Thus \eqref{PositivD} hold by $\S$6.2 and $\S$6.11 in \cite{T} and implies, together with Lemma \ref{Lemma2}, that the operator $\mathcal{V}_{\alpha;\Gamma}$ is $\widetilde{H}^{-\frac{1}{2}}(\Gamma_{D},\mathbb{R}^{2})$-elliptic.
\end{proof}

\begin{remark} 
\item[$\ \ (i)$] The hypersingular potential operator for the Brinkman system $\textbf{D}_{\alpha;\Gamma}$ is coercive on the subspace $H^{\frac{1}{2}}_{\nu}(\Gamma, \mathbb{R}^{2})$, in particular also on the space $\widetilde{H}^{\frac{1}{2}}(\Gamma_{N}, \mathbb{R}^{2})$ (see \cite[Theorem 3.8]{K-W} in the case of the Stokes system).
\item[$ \ \ (ii) \!$] The single-layer integral operator $\mathcal{V}_{\alpha;\Gamma}$ is coercive only on the subspace \\ $\widetilde{H}^{-\frac{1}{2}}(\Gamma_{D}, \mathbb{R}^{2})$, although it satisfies a Garding inequality for the Sobolev space defined on the entire $\Gamma$.
\end{remark}

\begin{theorem}
\label{Th2.3}
Let $ \mathfrak{D} \subset \mathbb{R}^{2}$ be a bounded Lipschitz domain with connected boundary $ \Gamma $ as in definition \ref{Domain}. Then the variational problem \eqref{Var Brinkman system} has a unique solution.
\end{theorem}

\begin{proof}
The proof of this theorem follows similar arguments to those of \cite[Theorem 3.10]{K-W}. Considering $ (\varphi_{N},\psi_{D}) \in \mathcal{H} $ we obtain the following relation for the bilinear form (\ref{bilinear form})
\begin{align}
a((\varphi_{N},\psi_{D});(\varphi_{N},\psi_{D})) &= \langle \mathcal{V}_{\alpha;\Gamma} \psi_{D}, \psi_{D} \rangle - \langle \textbf{K}_{\alpha;\Gamma} \varphi_{N} ,\psi_{D} \rangle \nonumber \\ &+ \langle \textbf{K}^{*}_{\alpha;\Gamma} \psi_{D},\varphi_{N} \rangle + \langle \textbf{D}_{\alpha; \Gamma} \varphi_{N}, \varphi_{N} \rangle \nonumber \\
&=\langle \mathcal{V}_{\alpha;\Gamma} \psi_{D}, \psi_{D} \rangle + \langle \textbf{D}_{\alpha; \Gamma} \varphi_{N}, \varphi_{N} \rangle.
\end{align}

Due to Lemma \ref{Lemma3} and Lemma \ref{Lemma2}, the bilinear form (\ref{bilinear form}) satisfies a Garding inequality of the form
\begin{equation}
(a+\mathcal{C})((\varphi_{N},\psi_{D});(\varphi_{N},\psi_{D})) \! \geq \!c_{D} \| \varphi_{N} \|^{2}_{H^{\frac{1}{2}}(\Gamma_{N}, \mathbb{R}^{2})} \!+\! c_{V} \| \psi_{D} \|^{2}_{H^{-\frac{1}{2}}(\Gamma_{D}, \mathbb{R}^{2})} \nonumber 
\end{equation}
with the bilinear form $\mathcal{C}$ given in terms of the compact operators $\mathcal{C}_{V}$ and $\mathcal{C}_{D}$ defined in \eqref{cmpV} and \eqref{cmp3} as follows
\begin{equation}
\mathcal{C}((u,t);(\varphi,\psi)) = \langle \mathcal{C}_{V}u,\varphi \rangle +  \langle \mathcal{C}_{D}t,\psi  \rangle, \quad \forall (\varphi,\psi) \in \mathcal{H}.
\end{equation}

%Therefore, the bilinear form $ a $ is $ \mathcal{H}-$elliptic. Hence the Fredholm alternative applies to (\ref{bilinear form}) and the uniqueness of the solution will imply the existence. In order to prove the existence of a solution we will show that the system (\ref{Eq1}) and (\ref{Eq2}) has a unique solution  $ (\varphi_{N},\psi_{D}) \in \mathcal{H} $.

Next, we show the positivity of the bilinear form $a$. To this end, let us consider the homogeneous system
\begin{equation}
\left\{
\begin{array}{l}
\label{homEq1}
\mathcal{V}_{\alpha;\Gamma} \psi^{0}_{D} - \textbf{K}_{\alpha;\Gamma} \varphi^{0} _{N}  =  0, \quad x \in \Gamma_{D} \\
\textbf{K}^{*}_{\alpha;\Gamma} \psi^{0} _{D}  + \textbf{D}_{\alpha; \Gamma}\varphi^{0}_{N}  =  0, \quad x \in \Gamma_{N}
\end{array}
\right.
\end{equation}
and let $ (\textbf{u}_{0}, p_{0}) $ be the fields given by 
\begin{equation}
\label{direct Approach2}
\textbf{u}(x) = \textbf{V}_{\alpha;\Gamma} (\psi^{0}_{D}) - \textbf{W}_{\alpha; \Gamma} (\varphi^{0} _{N}), \quad
p(x) = Q^{s}_{\alpha;\Gamma}  (\psi^{0}_{D}) - Q^{d}_{\alpha; \Gamma}  (\varphi^{0} _{N}),\ \mbox{ in } \ \mathbb{R}^{2}\backslash \Gamma.
\end{equation}
By the choice of the boundary conditions, we find that the interior boundary data satisfies 
\begin{equation}
\rm{Tr}^{+} \textbf{u}_{0}|_{\Gamma_{D}} = 0 \ \mbox{and} \ \partial^{+}_{\nu,\alpha}(\textbf{u}_{0},p_{0})|_{\Gamma_{N}} =0
\end{equation}
and hence by applying the Green formula \eqref{Green formula}, we obtain $ \textbf{u}_{0} = 0 $ in $ \mathfrak{D} $.

By the jump relations across $ \Gamma $ for the conormal derivative and the fact that $(\psi^{0}_{D})|_{\Gamma_{N}} = 0 $, we get
\begin{equation}
\partial^{-}_{\nu;\alpha}(\textbf{u}_{0},p_{0})|_{\Gamma_{N}} = \textbf{K}^{*}_{\alpha;\Gamma} \psi^{0} _{D}  + \textbf{D}_{\alpha; \Gamma}\varphi^{0}_{N}  =  0,
\end{equation}
and similarly 
\begin{equation}
\rm{Tr}^{-} \textbf{u}_{0}|_{\Gamma_{D}} = \mathcal{V}_{\alpha;\Gamma} \psi^{0}_{D} - \textbf{K}_{\alpha;\Gamma} \varphi^{0} _{N}  =  0.
\end{equation}
Because $ (\textbf{u}_{0},p_{0})  $ has to vanish at infinity  we have $ \textbf{u}_{0} = 0 $ in $ \mathfrak{D}_{-} $. The jump relations for the double-layer potential yield that $  0 = \rm{Tr}^{-} \textbf{u}_{0} - \rm{Tr}^{+} \textbf{u}_{0} = \varphi^{0}_{N} $ and for the single layer potential $ 0 = \partial^{-}_{\nu,\alpha}(\textbf{u}_{0},p_{0}) - \partial^{+}_{\nu,\alpha}(\textbf{u}_{0},p_{0}) = \psi^{0}_{D} $, i.e., $ (\varphi^{0}_{N}, \psi^{0}_{D}) = (0,0) $. Therefore, the bilinear form $ a $ is strongly $ \mathcal{H}-$elliptic and by the Lax-Milgram theorem (see, e.g., \cite[Theorem 5.2.3]{H-W}), the variational problem \ref{Var Brinkman system} has a unique solution.

Next, we show that the solution $(\textbf{u},p)$ is bounded. For any fixed $(\varphi_{N}, \psi_{D})$, we obtain a bounded, linear functional  $a_{( \varphi_{N},\psi_{D} )} \in \mathcal{H}'$ acting on $( \varphi, \psi) \in \mathcal{H}$ and hence by the Riesz representation theorem, there exists a unique $(f,g) \in \mathcal{H}$ such that 
\begin{equation}
a(( \varphi_{N}, \psi_{D});(\varphi,\psi)) = \langle (f,g),(\varphi,\psi) \rangle_{\mathcal{H}}, \quad \forall (\varphi,\psi) \in \mathcal{H}
\end{equation}
and we define $(f,g) = \mathcal{A}( \varphi_{N}, \psi_{D})$. It follows that the problem \eqref{Var Brinkman system} is equivalent to
\begin{equation}
\label{eqVarProb}
\mbox{Find} \ ( \varphi_{N}, \psi_{D}) \in \mathcal{H} \ \ \mbox{such that}  \ \
\mathcal{A}( \varphi_{N}, \psi_{D}) = (f,g).
\end{equation}

Since the bilinear form \eqref{bilinear form} is continuous, the operator $\mathcal{A} : \mathcal{H} \to \mathcal{H}$  and the functional $l : \mathcal{H} \to \mathbb{R}$ are linear and bounded. Moreover, the well-posedness of the variational problem \eqref{Var Brinkman system} implies that the operator $\mathcal{A}$ is invertible (for more details we refer the reader to \cite[Theorem 5.2.3]{H-W}). Thus, we obtain the estimate
\begin{equation}
\label{first estimate}
\|(\varphi_{N}, \psi_{D}) \|_{\mathcal{H}} = \|\mathcal{A}^{-1}(f,g) \|_{\mathcal{H}} \leq \|\mathcal{A}^{-1} \| \| (f,g) \|_{\mathcal{H}} \leq c\| (f,g) \|_{\mathcal{H}} 
\end{equation}

The Riesz representation operator $J : \mathcal{H}' \to \mathcal{H}$ mapping $\mathcal{H}' \ni l \to Jl = (f,g) \in \mathcal{H}$, i.e.,
\begin{equation}
\langle l, ( \varphi,\psi) \rangle = \langle Jl, ( \varphi, \psi) \rangle_{\mathcal{H}} := \langle (f,g), ( \varphi, \psi) \rangle_{\mathcal{H}}, \ \forall (\varphi,\psi) \in \mathcal{H},
\end{equation} 
is an isomorphism. Thus, the solution of the variational problem \eqref{Var Brinkman system} is given by
\begin{equation}
\label{PDVN}
( \varphi_{N}, \psi_{D} ) = \mathcal{A}^{-1} J l =: (\Phi J l, \Psi J l).
\end{equation}

Consequently, since the solution of the boundary value problem (\ref{Var Brinkman system}) is expressed in terms of layer potentials (\ref{direct Approach}) and due to the relation \eqref{PDVN}, we conclude that the solution is given by
\begin{align}
\label{SolDN}
\textbf{u}(x) = \textbf{V}_{\alpha;\Gamma} (\tilde{\textbf{g}} + \Psi J l ) - \textbf{W}_{\alpha; \Gamma} (\tilde{\textbf{f}} + \Phi J l)  =: \mathcal{U}(\textbf{f},\textbf{g})\\
p(x) = Q^{s}_{\alpha;\Gamma}  (\tilde{\textbf{g}} + \Psi J l ) - Q^{d}_{\alpha; \Gamma}  (\tilde{\textbf{f}} + \Phi J l )  =: \mathcal{P}(\textbf{f},\textbf{g}) 
\end{align}
and satisfies the estimate
\begin{equation}
\label{Estim}
\| \textbf{u} \| _{H^{1}(\mathfrak{D},\mathbb{R}^{2})} + \| p \|_{L^{2}(\mathfrak{D})} \leq c \big( \|\textbf{f}\|_{H^{\frac{1}{2}}(\Gamma_{D},\mathbb{R}^{2})} + \| \textbf{g} \| _{H^{-\frac{1}{2}}(\Gamma_{N},\mathbb{R}^{2})} \big).
\end{equation}
\end{proof}

\begin{remark} 
\item[$\ \ (i)$] Considering that $\mbox{meas}(\Gamma_{N}) = 0$, hence $\Gamma_{D} \equiv \Gamma$, we obtain the Dirichlet problem for the Brinkman system. Some attention is required since the single-potential integral operator $\mathcal{V}_{\alpha;\Gamma}$ is invertible only on the subspace $H^{\frac{1}{2}}_{\nu}(\Gamma, \mathbb{R}^{2})$, given by \eqref{OrthSubspace}. Thus, the operator (see, e.g., \cite[Theorem 5.2]{kohrpotanal})
\begin{align}
 \mathcal{V}_{\alpha;\Gamma} : H^{-\frac{1}{2}}(\Gamma, \mathbb{R}^{2})\slash \mathbb{R}\nu \to H^{\frac{1}{2}}_{\nu}(\Gamma, \mathbb{R}^{2}),
\end{align} 
is an isomorphism, and then the well-posedness result holds with the additional assumption that $\textbf{f} \in H^{\frac{1}{2}}_{\nu}(\Gamma, \mathbb{R}^{2})$ (see, e.g, \cite[Theorem 5.1 (iii)]{GuttMMA} and \cite[Theorem 9.18]{M-W} for the Stokes system). 
\item[$ \ \ (ii) \!$] Letting $\mbox{meas}(\Gamma_{D}) = 0$, i.e., $\Gamma_{N} \equiv \Gamma$, we obtain the Neumann problem for the Brinkman system (see, e.g, \cite[Theorem 5.6]{GuttMMA} and \cite[Theorem 9.19]{M-W} for the Stokes system).
\end{remark}

\subsection{The Poisson problem for the Brinkman system with mixed Dirichlet-Neumann conditions}
In this section we show the well-posedness result of the weak solution for the mixed boundary value problem of Dirichlet-Neumann type for the Brinkman system in $L^{2}$-based Sobolev spaces defined on a bounded Lipschitz domain $\mathfrak{D}$ in $\mathbb{R}^{2}$ with connected boundary. This immediate consequence of the above well-posedness result,  follows similar arguments as in  \cite{K-L-W4}, where the authors have studied the Poisson problem for the Stokes and Brinkman system (see also \cite[Section 4]{Kohr2016}).

For simplicity of notation, let us define the solution space $\mathcal{X}$, the boundary value space $\mathcal{B}_{DN}$ and the space $\mathcal{Y}$ for the mixed boundary value problem for the Brinkman system as
\begin{align}
\label{spacesXBY}
\mathcal{X}&:=H^{1}(\mathfrak{D},\mathbb{R}^{2})\times L^{2}(\mathfrak{D}),\ \ \mathcal{B}_{DN} := H^{\frac{1}{2}}(\Gamma_{D},\mathbb{R}^{2})\times H^{-\frac{1}{2}}(\Gamma_{N}, \mathbb{R}^{2}), \nonumber \\
\mathcal{Y}&:= \widetilde{H}^{-1}(\mathfrak{D},\mathbb{R}^{2}) \times \mathcal{B}_{DN}.
\end{align}
 
\begin{theorem} 
\label{Th2.6}
Assume that $\mathfrak{D} \subset \mathbb{R}^{2}$  is a bounded Lipschitz domain with connected boundary $\Gamma$ as in definition \ref{Domain}. Let $\alpha > 0$ be a constant. Then for all given data $({\bf f},{\bf g})\in \mathcal{B}_{DN}$, the mixed Dirichlet-Neumann boundary value problem for the Brinkman system
\begin{equation}
\label{mixed Poisson Brinkman system}
\left\{
\begin{array}{l}
\triangle {\bf u} - \alpha {\bf u} - \nabla p = \mathcal{F},\
{\rm{div}}\ {\bf u} = 0\ \mbox{ in }\  \mathfrak{D}\\
\left({\rm{Tr}} \ {\bf u}\right)|_{\Gamma_D}={\bf f}\in H^{\frac{1}{2}}(\Gamma_{D},\mathbb{R}^{2}) \\
\left(\partial_{\nu;\alpha}({\bf u},p )\right)|_{\Gamma_{N}} ={\bf g}\in H^{-\frac{1}{2}}(\Gamma_{N},\mathbb{R}^{2}), 
\end{array}
\right.
\end{equation}
is well-posed, i.e., there exists a unique solution $({\bf u},p)\in \mathcal{X}$. Moreover, there exists a linear continuous operator $ \mathcal{A}_{\alpha} : \mathcal{Y} \to \mathcal{X} $ delivering the solution, which satisfies the inequality
\begin{equation}
\label{estimate-mixed}
\|{\bf u}\|_{H^{1}(\mathfrak{D}, \mathbb{R}^{2})} + \| p \|_{L^{2}(\mathfrak{D})}\leq C\left( \| \mathcal{F}\|_{\widetilde{H}^{-1}(\mathfrak{D},\mathbb{R}^{2})} +\|{\bf f}\|_{H^{\frac{1}{2}}(\Gamma_{D},\mathbb{R}^{2})}+ \|{\bf g}\|_{H^{-\frac{1}{2}}(\Gamma_{N}, \mathbb{R}^{2})}\right), \nonumber
\end{equation}
with some constant $C\equiv C(\alpha, \Gamma_{D},\Gamma_{N})>0$.
\end{theorem}
\begin{proof}
We construct the solution of the problem \eqref{mixed Poisson Brinkman system} as a combination between Newtonian potential operators $({\mathcal N}_{\alpha ;{\mathfrak D}}{\mathcal F}, {\mathcal Q}_{\alpha ;{\mathfrak D}}{\mathcal F})$ and the unknown solution $(\textbf{v}, q) $, which satisfies the mixed Dirichlet-Neumann problem \eqref{Brinkman system} in the form (see, e.g., \cite[Theorem 6.1]{K-L-W2}) 
\begin{align}
\label{Poisson-mixed-2}
{\bf u}={\mathcal N}_{\alpha ;{\mathfrak D}}{\mathcal F}+{\bf v},\ p ={\mathcal Q}_{\alpha ;{\mathfrak D}}{\mathcal F}+q,
\end{align}
where
\begin{align}
\label{Newtonian-1a-Lp}
&{\mathcal N}_{\alpha;\mathfrak{D}}:H^{-1}(\mathfrak{D},\mathbb{R}^{2})\to H^{1}(\mathfrak{D},\mathbb{R}^{2}),\ \ {\mathcal N}_{\alpha;\mathfrak{D}}{\bf w}:=-\left\langle \mathcal{G}^{\alpha}(\cdot,\cdot), {\bf w}\right\rangle _{{\mathfrak D}},\\
\label{Newtonian-1b-Lp}
&{\mathcal Q}_{\alpha;\mathfrak{D}}:H^{-1}(\mathfrak{D},\mathbb{R}^{2})\to L^{2}(\mathfrak{D}),\ \ {\mathcal Q}_{\alpha;\mathfrak{D}}{\bf w}:=-\left\langle {\Pi}^{\alpha}(\cdot,\cdot), {\bf w}\right\rangle _{{\mathfrak D}}.
\end{align}
The Newtonian velocity and pressure potential operators of the Brinkman system are linear and continuous pseudodifferential operators of order $-2$ and $-1$ (see also \cite[Lemma 3.2]{K-L-M-W} for further related mapping properties). In addition, the Newtonian potentials satisfy the relation $({\mathcal N}_{\alpha;\mathfrak{D}}{\mathcal F}, {\mathcal Q}_{\alpha;\mathfrak{D}}{\mathcal F}; {\mathcal F}) \in H^{1}(\mathfrak{D}, \mathfrak{L}_{\alpha})$.

Then by the construction in \eqref{Poisson-mixed-2} and the properties of the Newtonian potentials, the Poisson problem of mixed type \eqref{mixed Poisson Brinkman system} is reduced to the mixed Dirichlet-Neumann problem for the homogeneous Brinkman system, i.e., $(\textbf{v}, q) \in H^{1}(\mathfrak{D}, \mathfrak{L}_{\alpha})$ and $ \mathfrak{L}_{\alpha}(\textbf{v},q) = \textbf{0} $, with the boundary conditions 

\begin{align}
\label{modif}
{\bf f}_0:={\bf f}-\left({\mathcal N}_{\alpha ;{\mathfrak D}}{\mathcal F}\right)|_{\Gamma_D},\ \
{\bf g}_0:={\bf g}-\left(\partial _{\alpha ;\nu }\left({\mathcal N}_{\alpha ;{\mathfrak D}}{\mathcal F},{\mathcal Q}_{\alpha ;{\mathfrak D}}{\mathcal F}\right)\right)|_{\Gamma_{N}}.
\end{align}
Theorem \ref{Th2.3} asserts that the mixed Dirichlet-Neumann boundary value problem with boundary data $({\bf 0},{\bf f}_0,{\bf g}_0)$ has a unique solution $({\bf v},q)\in \mathcal{X},$ which satisfies an estimate of type \eqref{Estim}. Hence, we obtain the unique solution of the Poisson problem for the Brinkman system delivered by the linear and continuous operator  $\mathcal{A}_{\alpha} : \mathcal{Y} \to \mathcal{X}$ defined by
\begin{align}
 \mathcal{A}_{\alpha}(\mathcal{F}, \textbf{f}, \textbf{g}) = (\mathcal{U}_{\mathcal{F}}, \mathcal{P}_{\mathcal{F}}) := \big({\mathcal N}_{\alpha ;{\mathfrak D}}{\mathcal F} + \mathcal{U}(\textbf{f}_{0}, \textbf{g}_{0}) , Q_{\alpha ;{\mathfrak D}}{\mathcal F} + \mathcal{P}(\textbf{f}_{0}, \textbf{g}_{0}) \big).
\end{align}
\end{proof}

\subsection{The mixed Dirichlet-Robin problem for the Brinkman system}
Next, we are concerned with the mixed Dirichlet-Robin boundary value problem for the Brinkman system. Let us consider now that the boundary $\Gamma$ is partitioned in two non-overlapping parts $\Gamma_{D}$ and $\Gamma_{R}$ such that $\overline{\Gamma}_{D} \cup \overline{\Gamma}_{R} = \Gamma$ in analogy with definition \ref{Domain}, i.e., now we have $\mathcal{B}_{DR} = H^{\frac{1}{2}}(\Gamma_{D},\mathbb{R}^{2}) \times H^{-\frac{1}{2}}(\Gamma_{R},\mathbb{R}^{2})$ and $\mathcal{Y} = \widetilde{H}^{-1}(\mathfrak{D}, \mathbb{R}^{2}) \times \mathcal{B}_{DR}$. Then the mixed Dirichlet-Robin boundary value problem for the Brinkman system is
\begin{equation}
\label{D-R Brinkman system}
\left\{
\begin{array}{l}
\Delta \textbf{u} - \alpha \textbf{u} - \nabla p = \mathcal{F}, \
\rm{div} \ \textbf{u} = 0  \ \mbox{in} \ \mathfrak{D}  \\
(\rm{Tr} \ \textbf{u})|_{\Gamma_D} = \textbf{f} \in H^{\frac{1}{2}}(\Gamma_{D},\mathbb{R}^{2}), \\
(\partial_{\nu;\alpha}(\textbf{u},p)|_{\Gamma_{R}}  +   (\lambda \rm{Tr} \ \textbf{u})|_{\Gamma_{R}}  = \textbf{h}  \in H^{-\frac{1}{2}}(\Gamma_{R},\mathbb{R}^{2}).
\end{array}
\right.
\end{equation}
where $(\cdot)|_{\Gamma_{R}}$ denotes the restriction operator from the entire boundary to $\Gamma_{R}$ and $\lambda \in L^{\infty}(\Gamma_{R},\mathbb{R}^{2} \otimes \mathbb{R}^{2})$ is a symmetric matrix valued function such that (as in \cite[Theorem 4.1]{K-L-W3})
\begin{equation}
\label{lambda}
\langle \lambda \textbf{v}, \textbf{v} \rangle_{\Gamma_{R}} \geq 0, \ \forall {\bf v} \in L^{2}(\Gamma_{R}, \mathbb{R}^{2}).
\end{equation}

\begin{theorem}
\label{Th2.7}
Let $ \mathfrak{D} \subset \mathbb{R}^{2}$ be a bounded Lipschitz domain with connected boundary $ \Gamma = \partial \mathfrak{D} $, which is decomposed into two adjacent, nonoverlapping parts $ \Gamma = \overline{\Gamma}_{D} \cup \overline{\Gamma}_{R} $ as in definition \ref{Domain}. Let $\alpha >0$ be given constants and let $\lambda \in L^{\infty}(\Gamma,\mathbb{R}^{2} \otimes \mathbb{R}^{2})$  be a symmetric matrix valued function with property \eqref{lambda}. Then the problem \eqref{D-R Brinkman system} has a unique solution, which satisfies an estimate 
\begin{equation}
\| {\bf u} \| _{H^{1}(\mathfrak{D},\mathbb{R}^{2})} + \| p \|_{L^{2}(\mathfrak{D})} \leq c \big(\|\mathcal{F} \|_{\widetilde{H}^{-1}(\mathfrak{D},\mathbb{R}^{2})} + \| {\bf f}\|_{H^{\frac{1}{2}}(\Gamma_{D},\mathbb{R}^{2})} + \| {\bf h} \| _{H^{-\frac{1}{2}}(\Gamma_{R},\mathbb{R}^{2})} \big)
\end{equation}
\end{theorem}

\begin{proof} The proof of this theorem follows the main ideas in the proof of \cite[Theorem 4.4]{Kohr2016}.
Theorem \ref{Th2.6} asserts that there exists a linear and continuous operator $\mathcal{A}_{\alpha} : \mathcal{Y} \to \mathcal{X} $ delivering the unique solution $(\textbf{u}, p)$ of the mixed Dirichlet-Neumann problem
\begin{equation}
(\textbf{u}, p) = \mathcal{A}_{\alpha}(\mathcal{F}, \textbf{f},\textbf{h} - (\lambda \rm{Tr} \ \textbf{u} )|_{\Gamma_{R}}),
\end{equation}
which can be rewritten, due to the linearity of the operator $\mathcal{A}_{\alpha}$, as
\begin{equation}
\label{EqDN-DR}
(\textbf{u}, p) + \mathcal{A}_{\alpha}(\textbf{0}, \textbf{0}, (\lambda \rm{Tr} \ \textbf{u} )|_{\Gamma
_{R}} ) = \mathcal{A}_{\alpha}(\mathcal{F}, \textbf{f},\textbf{h} ).
\end{equation}
The compactness of the embedding
\begin{equation}
L^{\infty}(\Gamma, \mathbb{R}^{2} \otimes \mathbb{R}^{2}) \cdot H^{\frac{1}{2}}(\Gamma, \mathbb{R}^{2}) \hookrightarrow L^{2}(\Gamma, \mathbb{R}^{2}) 
\end{equation}
together with the continuity of the restriction to $\Gamma_{R}$, the continuity of the embedding
\begin{equation}
L^{2}(\Gamma_{R}, \mathbb{R}^{2}) \hookrightarrow H^{-\frac{1}{2}}(\Gamma_{R}, \mathbb{R}^{2})
\end{equation}
and the continuity of $\mathcal{A}_{\alpha}$, implies that the left hand side of \eqref{EqDN-DR} defines a Fredholm operator with index zero denoted by
\begin{equation}
\mathcal{B}_{\lambda} : H^{1}(\mathfrak{D},\mathbb{R}^{2}) \times L^{2}(\mathfrak{D})  \to H^{1}(\mathfrak{D},\mathbb{R}^{2}) \times L^{2}(\mathfrak{D}). 
\end{equation}
Next, we show that the operator $\mathcal{B}_{\lambda}$ is one-to-one, i.e., that $\rm{Ker} \ \mathcal{B}_{\lambda} = \{ (\textbf{0},0)\}$, which is equivalent to the well-posedness of the mixed boundary value problem \eqref{D-R Brinkman system}.

To this end, let us assume that the homogeneous problem associated to \eqref{D-R Brinkman system} has the solution
$(\textbf{u}_{0},p_{0}) \in \mathcal{X} $. Applying Green identity \eqref{Green formula} to $\mathfrak{D}$, we obtain (see, e.g., \cite[Theorem 6.1]{K-L-W2} )
\begin{align}
\label{DREQ3}
 2\langle \mathbb{E}(\textbf{u}_{0}), \mathbb{E}(\textbf{u}_{0}) \rangle  +  \alpha \langle \textbf{u}_{0}, \textbf{u}_{0} \rangle  = \langle \partial_{\nu; \alpha}(\textbf{u}_{0}, p_{0}), \rm{Tr} \ \textbf{u}_{0} \rangle 
= - \langle (\lambda \rm{Tr} \ \textbf{u}_{0})|_{\Gamma_{R}}, (\rm{Tr} \ \textbf{u}_{0})|_{\Gamma_{R}} \rangle. 
\end{align}
The left hand side of \eqref{DREQ3}  is non-negative and the right hand side is non-positive, due to condition \eqref{lambda} for the tensor field $\lambda$. Hence, each term of \eqref{DREQ3} vanishes and we obtain $\textbf{u} = 0$ and $p_{0} = c \in \mathbb{R}$ in $\mathfrak{D}$. In addition, $(\rm{Tr} \ \textbf{u})|_{\Gamma_{R}} = 0$, which implies that $(\partial_{\nu;\alpha}(\textbf{u}, p))|_{\Gamma_{R}} = 0$ such that we also obtain $p_{0} = 0$ in $\mathfrak{D}$ and hence the desired uniqueness result is proved.
Consequently, the solution of the mixed Dirichlet-Robin boundary value problem is given by the operator
\begin{equation}
\label{Salphalambda}
\mathcal{S}_{\alpha;\lambda} : \mathcal{Y} \to \mathcal{X}, \quad \mathcal{S}_{\alpha;\lambda} := \mathcal{B}^{-1}_{\lambda}\mathcal{A}_{\alpha} : \mathcal{Y} \to \mathcal{X}.
\end{equation}
\end{proof}

\subsection{Mixed Dirichlet-Robin problem for the nonlinear Darcy-Forchheimer-Brink\-man system }
Going further, we obtain a similar existence and uniqueness result as in \cite[Theorem 7.1]{K-L-W2} for the weak solution of the mixed Dirichlet-Robin problem \eqref{Darcy-Brinkamn problem}, with the given data $(\textbf{f}, \textbf{h}) \in \mathcal{B}_{DR}$.  The Darcy-Forchheimer-Brinkman system with Robin boundary conditions in Lipschitz domains in Euclidean settings has been investigated in \cite{K-L-W3} (see also \cite{K-L-W-Robin} and \cite{Kohr2016} for transmission problems). Recently, the authors\cite{GuttMMA} obtained well-posedness results for the mixed Dirichlet-Neumann problem for semilinear Darcy-Forchheimer-Brinkman system on creased Lipschitz domains in $\mathbb{R}^{3}$.

\begin{theorem}
\label{Theorem 3.2}
Let $\mathfrak{D} \subset \mathbb{R}^{2}$ be a bounded Lipschitz domain with connected boundary $\Gamma $, which is decomposed similar to definition \ref{Domain} into two disjoint nonoverlapping parts $\Gamma_{D}$ and $\Gamma_{R}$. Let $\alpha ,\beta >0$ be given constants and $\lambda \in L^{\infty}(\Gamma,\mathbb{R}^{2} \otimes \mathbb{R}^{2})$ is a symmetric matrix valued function with property \eqref{lambda}. Then there exists two constants $C_{j}\equiv C_{j}(\mathfrak{D},\alpha,\beta)>0$, $j=1,2$, with the property that for all data $({\bf f},{\bf h})\in \mathcal{B}_{DR}$ satisfying the condition
\begin{equation}
\label{inequality p2 1}
\|{\bf f}\|_{H^{\frac{1}{2}}(\Gamma_{D},\mathbb{R}^{2})}+\|{\bf h}\|_{H^{-\frac{1}{2}}(\Gamma_{R},\mathbb{R}^{2})}\leq C_{1},
\end{equation}
the mixed Dirichlet-Robin problem for the nonlinear Darcy-Forchheimer-Brinkman system
\begin{equation}
\label{Darcy-Brinkamn problem}
\left\{
\begin{array}{l}
\triangle {\bf u} - \alpha {\bf u} -  \beta ({ {\bf u} \cdot \nabla}){\bf u} - \nabla p ={\bf 0},\ {\rm{div}}\ {\bf u}=0\ \mbox{ in }\ {\mathfrak D},\\
\left({\rm{Tr}} \ {\bf u}\right)|_{\Gamma_{D}}={\bf f}  \in  H^{\frac{1}{2}}(\Gamma_{D}, \mathbb{R}^{2})  \\
\left(\partial_{\nu;\alpha }({\bf u},p )\right)|_{\Gamma_{R}} + \lambda \left({\rm{Tr}} \ {\bf u}\right)|_{\Gamma_{R}} ={\bf h} \in H^{-\frac{1}{2}}(\Gamma_{R}, \mathbb{R}^{2})
\end{array}
\right.
\end{equation}
has a unique solution $({\bf u},p)\in \mathcal{X}$, with the property $\|{\bf u}\|_{H^{1}(\mathfrak{D},\mathbb{R}^{2})}\leq C_{2}$. Moreover, the solution depends continuously on the given data, and satisfies the estimate
\begin{align}
\label{estimate-D-B-F-new1-new1-D-Rn2-Lp}
\|{\bf u}\|_{H^{1}(\mathfrak{D},\mathbb{R}^{2})}+\|p \|_{L^{2}(\mathfrak{D})}\leq C\left(\|{\bf f}\|_{H^{\frac{1}{2}}(\Gamma_{D},\mathbb{R}^{2})}+\|{\bf h}\|_{H^{-\frac{1}{2}}(\Gamma_{R},\mathbb{R}^{2})}\right)
\end{align}
with some constant $C\equiv C(\mathfrak{D},\alpha ,\beta)>0$.
\end{theorem}

\begin{proof}
We use similar arguments to those developed in the proof of \cite[Theorem 5.2]{K-L-M-W} devoted to transmission type problems in ${\mathbb R}^3$ for the Stokes and Darcy-Forchheimer-Brinkman systems in Sobolev spaces. 

By using the continuity of the embeddings
\begin{equation}
\label{embbeding}
H^{1}(\mathfrak{D},\mathbb{R}^{2}) \hookrightarrow L^{6}(\mathfrak{D},\mathbb{R}^{2}) \hookrightarrow L^{2}(\mathfrak{D},\mathbb{R}^{2}),
\end{equation}
and by the H\"{o}lder inequality, we obtain the estimates (\cite[Theorem 4.1]{Grosan})
\begin{align}
\label{3.0.3}
\| (\textbf{v} \cdot \nabla){\bf w}\ \|_{L^{\frac{3}{2}}(\mathfrak{D},\mathbb{R}^{2})}
\leq \|{\bf v}\|_{L^{6}(\mathfrak{D},\mathbb{R}^{2})}\|{\nabla \bf w}\|_{L^{2}(\mathfrak{D},\mathbb{R}^{2})} 
\leq c_0\|{\bf v}\|_{H^{1}(\mathfrak{D},\mathbb{R}^{2})}\|{\bf w}\|_{H^{1}(\mathfrak{D},\mathbb{R}^{2})},
\end{align}
with some constant $c_{0}\equiv c_{0}(\mathfrak{D})>0$. Then a duality argument based on \eqref{3.0.3}, we deduce that
\begin{equation}
\label{3.0.5-Lp}
({\bf v} \cdot \nabla ){\bf w} \in  L^{\frac{3}{2}}(\mathfrak{D},\mathbb{R}^{2}) \hookrightarrow \widetilde{H}^{-1}(\mathfrak{D},\mathbb{R}^{2}),\ \ \forall \ {\bf v},\, {\bf w}\in  H^{1}(\mathfrak{D},\mathbb{R}^{2}).
\end{equation}

Let us consider a fixed ${\mathbf v}\in H^{1}(\mathfrak{D},\mathbb{R}^{2})$ and the corresponding linear Poisson problem for the Brinkman system
\begin{equation}
\label{V-v0-1}
\left\{
\begin{array}{l}
\triangle {{\bf v}^0}-\alpha {{\bf v}^0}-\nabla {p^0}=
 \beta ({\bf v} \cdot \nabla ){\bf v}, \ \rm{div} \ \textbf{v}^{0} = 0, \ \mbox{ in }\ {\mathfrak D},\\ 
\left({\rm{Tr}} \ {{\bf v}^0}\right)|_{\Gamma_{D}}
={\bf f}\in H^{\frac{1}{2}}(\Gamma_{D},{\mathbb R}^2),\\
\left(\partial _{\nu ;\alpha }\left({{\bf v}^0},{p^0}\right)_{\beta ({\bf v} \cdot \nabla ){\bf v}}\right)|_{\Gamma_{R}}   + \lambda \left({\rm{Tr}}\ {\bf v}^{0}\right)|_{\Gamma_{R}} ={\bf h}\in H^{-\frac{1}{2}}(\Gamma_{R},{\mathbb R}^2),
\end{array}\right.
\end{equation}
with unknown fields $({\bf v}^0,p^0)\in \mathcal{X}$.
Theorem \ref{Th2.7} asserts that the problem \eqref{V-v0-1} with given data $\left(\textbf{v} \cdot \nabla \textbf{v},{\bf f},{\bf h}\right)\in \mathcal{Y}$ has a unique solution given by the continuous linear operator ${\mathcal S}_{\alpha;\lambda}:\mathcal{Y} \to \mathcal{X}$ of \eqref{Salphalambda}, i.e.,
\begin{align}
\label{solution-v0-Lp}
\left({\bf v}^0,p^0\right):=\left({\mathscr U}({\bf v}),{\mathscr P}({\bf v})\right)
={\mathcal S}_{\alpha;\lambda} \left(\beta ({\bf v} \cdot \nabla ){\bf v},\ {\bf f},\ {\bf h}\right)\in {\mathcal X},
\end{align}
Therefore, the nonlinear operators $({\mathscr U},{\mathscr P}):H^{1}(\mathfrak{D},\mathbb{R}^{2})\to {\mathcal X}$ defined for fixed boundary data $({\bf f},{\bf h})\in \mathcal{B}_{DR}$ satisfy the system
\begin{equation}
\label{Newtonian-D-B-F-4}
\left\{\begin{array}{lll}
\triangle {\mathscr U}({\bf v})-\alpha {\mathscr U}({\bf v})-\nabla {\mathscr P}({\bf v})=\beta(\textbf{v} \cdot \nabla){\bf v}\in H^{-1}({\mathfrak D},{\mathbb R}^{2}),\\
\left({\rm{Tr}}\left({{\mathscr U}({\bf v})}\right)\right)|_{\Gamma_{D}}={\bf f}\in H^{\frac{1}{2}}(\Gamma_{D},{\mathbb R}^{2}),\\
\left(\partial _{\nu ;\alpha }\left({\mathscr U}({\bf v}),{\mathscr P}({\bf v})\right)|_{\beta ({\bf v} \cdot \nabla ) {\bf v}}\right)\!|_{\Gamma_{R}}\!\!+ \lambda \left({\rm{Tr}}\left({{\mathscr U}({\bf v})}\right)\right)\!|_{\Gamma_{R}}\!\!={\bf h}\in H^{-\frac{1}{2}}(\Gamma_{R},{\mathbb R}^{2}).
\end{array}\right.
\end{equation}
and they are continuous and bounded, i.e., there is a constant\footnote{${\mathcal L}(\mathcal Y,\mathcal X):=\left\{T:\mathcal Y\to \mathcal X:T \mbox{ is a linear and continuous operator}\right\}$.} $c_{s}=\| \mathcal{S}_{\alpha;\lambda}\|_{{\mathcal L}(\mathcal Y,\mathcal X)}$ such that
\begin{align}
\label{estimate-D-B-F-3}
\big\|\big({\mathscr U}({\bf w}),{\mathscr P}({\bf w})\big)\big\|_{{\mathcal X}}&\leq c_{s}\|\left(\beta (\textbf{w} \cdot \nabla){\bf w},{\bf f},{\bf h}\right)\|_{{\mathcal Y}}\\
&\leq c_{s}\left(\|\left({\bf f},{\bf h}\right)\|_{\mathcal{B}_{DR}}+\beta \|\ (\textbf{w} \cdot \nabla){\bf w}\ \|_{L^{\frac{3}{2}}({\mathfrak D},{\mathbb R}^{2})}\right)\nonumber\\
&\leq c_{s}\|\left({\bf f},{\bf h}\right)\|_{\mathcal{B}_{DR}}
+c_{s}c_{0}\beta\|{\bf w}\|_{H^{1}(\mathfrak{D},\mathbb{R}^{2})}^2,
\ \ \forall\ {\bf w}\in H^{1}(\mathfrak{D},\mathbb{R}^{2}),\nonumber
\end{align}
with the constant  $c_{0}\equiv c_0({\mathfrak D})>0$ from inequality \eqref{3.0.3}. 

Thus, a fixed point ${\bf u}\in H^{1}(\mathfrak{D},\mathbb{R}^{2})$ of the nonlinear operator $\mathscr{U}$   together with the pressure function $p ={\mathscr P}({\bf u})$ determine a solution of the nonlinear mixed problem \eqref{Darcy-Brinkamn problem} in the space ${\mathcal X}$.
First, we determine a radius $\eta_{p}$, i.e., a closed ball $\mathbf{B}_{\eta} :=\{{\bf w}\in H^{1}(\mathfrak{D},\mathbb{R}^{2}):\|{\bf w}\|_{H^{1}(\mathfrak{D},\mathbb{R}^{2})}\leq \eta \}$ for which the operator $\mathscr{U}$ is invariant. To this end, let us consider the constants
\begin{align}
\label{Newtonian-D-B-F-new9n-Lp}
&\zeta:=\frac{2}{9\beta c_{0}c_{s}^2}>0,\ \ \eta:=\frac{1}{3\beta c_{0}c_{s}}>0
\end{align}
and assume that the given data satisfy the inequality $\|\left({\bf f},{\bf g}\right)\|_{\mathcal{B}_{DR}}\leq \zeta$. Introducing these estimates into (\ref{estimate-D-B-F-3}), we deduce that
\begin{align}
\label{Newtonian-D-B-F-new9-new-crackn-Lp}
\|\left({{\mathscr U}}({\bf w}),{\mathscr P}({\bf w})\right)\|_{{\mathcal X}}\leq \frac{1}{3\beta c_{0}c_{s}}= \eta,\ \forall \ {\bf w}\in \mathbf{B}_{\eta },
\end{align}
i.e., $\|{\mathscr U}({\bf w})\|_{H^{1}(\mathfrak{D},\mathbb{R}^{2})}\leq \eta$ for any ${\bf w}\in \mathbf{B}_{\eta}$. Consequently, ${\mathscr U}$ maps $\mathbf{B}_{\eta}$ to $\mathbf{B}_{\eta}$.

Next, the linearity and continuity of the operator ${\mathcal S}_{\alpha;\lambda}$, and inequality \eqref{3.0.3}, shows that the operator $\mathcal{U}$ is a contraction, 
\begin{align}
\label{4.3}
\|{\mathscr U}({\bf v})-{\mathscr U}({\bf w})\|_{H^{1}(\mathfrak{D},\mathbb{R}^{2})}
&\leq \|{\mathcal S}_{\alpha;\lambda}\left(\beta ({\bf v} \cdot \nabla){\bf v}-\beta ({\bf w} \cdot \nabla){\bf w},\ {\bf 0},\ {\bf 0}\right)\|_{H^{1}(\mathfrak{D},\mathbb{R}^{2})}\nonumber\\
&\leq c_{s}\beta \|\ ({\bf v} \cdot \nabla){\bf v}-({\bf w} \cdot \nabla){\bf w}\ \|_{\widetilde{H}^{-1}({\mathfrak D},{\mathbb R}^{2})}\nonumber
\\
&\leq c_{s}\beta \|\ (({\bf v} - {\bf w}) \cdot \nabla){\bf v}+({\bf w} \cdot \nabla)({\bf v} - {\bf w} )\ \|_{\widetilde{H}^{-1}({\mathfrak D},{\mathbb R}^{2})}  \nonumber 
\\
& \leq 2\eta c_{s}c_{0} \beta\|{\bf v}-{\bf w}\|_{H^{1}(\mathfrak{D},\mathbb{R}^{2})}\nonumber
\\
&=\frac{2}{3}\|{\bf v}-{\bf w}\|_{H^{1}(\mathfrak{D},\mathbb{R}^{2})},
\ \forall\ {\bf v},{\bf w}\in \mathbf{B}_{\eta}.
\end{align}
where $c_{s} = \| {\mathcal S}_{\alpha;\lambda} \| $ as in \eqref{estimate-D-B-F-3} and $c_{0}$ as in \eqref{3.0.3}. Hence, ${\mathscr U}$ is a contraction on $\mathbf{B}_{\eta}$.
Therefore the Banach-Caccioppoli fixed point theorem implies that there exists a unique fixed point ${\bf u}\in \mathbf{B}_{\eta}$ of ${\mathscr U}$, i.e., ${\mathscr U}({\bf u})={\bf u}$, which determines together with the pressure field $p ={\mathscr P}({\bf u})$, given by \eqref{solution-v0-Lp}, a solution of the nonlinear problem \eqref{Darcy-Brinkamn problem} in $\mathcal{X}$.

Since the solution satisfies the condition ${\bf u}\in \mathbf{B}_{\eta }$, by inequality \eqref{estimate-D-B-F-3} we obtain the estimate
\begin{align}
\label{estimate-D-B-F-5}
\|{\bf u}\|_{H^{1}(\mathfrak{D},\mathbb{R}^{2})}+
\|p\|_{L^{2}(\mathfrak{D})}\leq c_{s}\big\|\big({\bf f},{\bf h}\big)\big\|_{\mathcal{B}_{DR}}
+\frac{1}{3}\|{\bf u}\|_{H^{1}(\mathfrak{D},\mathbb{R}^{2})},
\end{align}
which implies that $ \|{\bf u}\|_{H^{1}(\mathfrak{D},\mathbb{R}^{2})}\leq
\dfrac{3}{2}c_{s}\|({\bf f},{\bf h})\|_{\mathcal{B}_{DR}} $.  This estimates together with \eqref{estimate-D-B-F-5}, yields
\begin{align}
\|{\bf u}\|_{H^{1}(\mathfrak{D},\mathbb{R}^{2})}+
\|p\|_{L^{2}({\mathfrak D})}\leq \frac{3}{2}c_{s}\big\|\big({\bf f},{\bf h}\big)\big\|_{\mathcal{B}_{DR}},\nonumber
\end{align}
i.e., just the inequality \eqref{estimate-D-B-F-new1-new1-D-Rn2-Lp} with the constant $C\equiv \dfrac{3}{2}c_{s}=\dfrac{3}{2}\|{\mathcal S}_{\alpha;\lambda}\|_{{\mathcal L}(\mathcal Y,\mathcal X)}$.

In order to prove the uniqueness of the solution $\left({\bf u},p \right)\in {\mathcal X}$, satisfying $\textbf{u} \in \mathbf{B}_{\eta} $, we consider that $\left({\bf u}',p '\right)\in \mathcal{X}$ is another solution of the problem \eqref{Darcy-Brinkamn problem}, such that $\textbf{u}' \in \mathbf{B}_{\eta} $.
Hence, ${\mathscr U}({\bf u}')\in \mathbf{B}_{\eta}$, such that $\left({\mathscr U}({\bf u}'),{\mathscr P}({\bf u}')\right)$  satisfy (\ref{Newtonian-D-B-F-4}) with $\textbf{u}'$ in place on $\textbf{v}$. Substracting the two systems in $(\textbf{u}',p')$ and $(\mathcal{U}(\textbf{u}'), \mathcal{P}(\textbf{u}'))$, we obtain the linear problem given by 
\begin{equation}
\left(\left({\mathscr U}({\bf u}')-{\bf u}'\right),\left({{\mathscr P}({\bf u}')}-p '\right)\right)={\mathcal S}_{\alpha;\lambda}({\mathbf 0},{\mathbf 0},{\mathbf 0}),
\end{equation}
which has only the trivial solution in the space ${\mathcal X}$, i.e., $\textbf{u}'$ is a fixed point of the operator
${\mathscr U}$.  Since ${\mathscr U}:\mathbf{B}_{\eta }\to \mathbf{B}_{\eta }$ is a contraction, it has a unique fixed point $\textbf{u}$ in $\mathbf{B}_{\eta}$. Consequently, $\textbf{u}'=\textbf{u}$, and, in addition, $p'=p $.
\end{proof}

\section{Application of DRBEM}
In Section 2 we have obtained the well-posedness result of the Dirichlet-Robin boundary value problem for the nonlinear Darcy-Forchheimer-Brinkman system (\ref{Darcy-Brinkamn problem}) in a two-dimensional Lipschitz domain $\mathfrak{D} \subset \mathbb{R}^{2} $. Next, we consider a numerical application of Theorem \ref{Theorem 3.2}, i.e., the lid-driven cavity flow problem. 

\subsection{Vorticity-streamfuntion approach for the nonlinear Darcy-Forchheimer-Brink\-man system}
The nonlinear Darcy-Forchheimer-Brinkman system \eqref{Darcy-Brinkamn problem} describes the flow of an incompressible, viscous fluid located in a square cavity filled with a saturated porous domain. Under the tangential motion of the upper wall, the fluid inside the cavity rotates until it arrives at an equilibrium state that is described by the equations in (\ref{Darcy-Brinkamn problem}). The numerical problem under investigation is solved by using the Dual Reciprocity Boundary Element Method (DRBEM). 

Let $ \textbf{u} (x,y) = (u(x,y),v(x,y)) $ be the velocity field of the flow and $ p = p(x,y) $ the corresponding pressure field. Then the projection of the nonlinear Darcy-Forchheimer-Brinkman consists of the following equations
\begin{equation}
\left\{
\begin{split}
\label{Ox}
&Ox:\quad \Big(\frac{\partial^{2} u}{\partial x^{2}} + \frac{\partial^{2} u}{\partial y^{2}}\Big) - \alpha u - \frac{\partial p}{\partial x}= \beta \Big( u\frac{\partial u}{\partial x} + v \frac{\partial u}{\partial y}\Big), \\
&Oy:\quad \Big(\frac{\partial^{2} v}{\partial x^{2}} + \frac{\partial^{2} v}{\partial y^{2}}\Big) - \alpha v - \frac{\partial p}{\partial y}= \beta \Big( u\frac{\partial v}{\partial x} + v \frac{\partial v}{\partial y}\Big), 
\end{split}
\quad  \frac{\partial u}{\partial x} + \frac{\partial v}{\partial y} = 0.
\right.
\end{equation}

The structure of the fluid flow through the porous media is related to four parameters that describe the physical properties of the fluid and the media, which are \textit{the Reynolds number} $Re$, \textit{the Darcy number} $Da$, \textit{the porosity} $\phi$ and the \textit{viscosity coefficient} $\Lambda$ (for the physical description of these parameters see, e.g., \cite{Ni-Be}), which are related to the mathematical parameters $\alpha$ and $\beta$ through the relations:
\begin{equation}
\alpha = \dfrac{\phi}{Da \Lambda}, \qquad \beta = \dfrac{Re}{\phi \Lambda}.
\end{equation}

In order to implement the mathematical model from Theorem \ref{Theorem 3.2}, we consider the stream function $\Psi$ and the vorticity $\omega$ defined for two-dimensional flows given by
\begin{equation}
\label{stream-vor}
\frac{\partial \Psi}{\partial x} = -v, \ \frac{\partial \Psi}{\partial y} = u, \quad  \mbox{respectively}  \quad  \omega = \frac{\partial v}{\partial x}  -\frac{\partial u}{\partial y}. 
\end{equation}

For simplicity, we consider a non-dimensional form for the nonlinear Darcy-Forchheimer-Brinkman system, i.e., the square cavity has dimension one and the velocity of the moving wall is scaled to unity (for details the reader is referred to \cite[Section 4.1]{Gutt}). Eliminating the pressure field from equations \eqref{Ox} and introducing the stream function and the vorticity (see, e.g., \cite[pp. 307-309]{Ni-Be}), we obtain the following system 
\begin{equation}
\label{non-System}
\left\{
\begin{array}{l}
\dfrac{\partial^{2} \omega}{\partial X^{2}} + \dfrac{\partial^{2} \omega}{\partial Y^{2}} =  \alpha \omega + \beta \Bigg[\dfrac{\partial \Psi}{\partial Y}\dfrac{\partial \omega}{\partial X} - \dfrac{\partial \Psi}{\partial X}\dfrac{\partial \omega}{\partial Y} \Bigg] \\
\dfrac{\partial^{2} \Psi}{\partial X^{2}} + \dfrac{\partial^{2} \Psi}{\partial Y^{2}} = -\omega,
\end{array}
\right.
\end{equation}
and the corresponding non-dimensional mixed Dirichlet-Robin boundary conditions
\begin{equation}
\label{Boundary}
\Psi = 0 \quad \mbox{on} \ \Gamma_{1}, \Gamma_{2} \mbox{ and } \Gamma_{3}, \qquad  \frac{\partial \Psi}{\partial Y} = 1 + S \frac{\partial^{2} \Psi}{\partial Y^{2}} \quad \mbox{on} \ \Gamma_{4}.
\end{equation}

\subsection{DRBEM Method}
The BEM transforms differential equations into integral equations defined on the boundary (and most of the times also on the domain). The non-dimensional form of the nonlinear Darcy-Forchheimer-Brinkman system \eqref{non-System} is weighted through the domain $\mathfrak{D}$, by the fundamental solutions of the Laplace operator in two dimensions (see, e.g., \cite{Brebbia-Telles})
\begin{equation}
u^{*} = \frac{1}{2}\log \left( \frac{1}{r} \right), \quad u_{q}^{*} = \frac{\partial u^{*}}{\partial \nu} 
\end{equation}
where $r = |\vec{r} - \vec{r}_{i}|$ is the distance function between the source point (fixed) and the field point (variable) \cite[Section 3.3.2]{Brebbia-Wrobel}. Note that the subscript $(\cdot)_{q}$ refers to the normal derivative. Applying Green's second identity to the above system, we obtain for each source point $i$ the relations 
\begin{equation}
\label{Bem1}
\left\{
\begin{array}{ll}
c_{i}\psi_{i} + \int_{\Gamma}u^{*}_{q}\psi d\Gamma - \int_{\Gamma}u^{*}\psi_{q} d\Gamma = \int_{\mathfrak{D}}(-\omega)d\mathfrak{D} \\
c_{i}\omega_{i} + \int_{\Gamma}u^{*}_{q}\omega d\Gamma - \int_{\Gamma}u^{*}\omega_{q} d\Gamma = \int_{\mathfrak{D}} \alpha \omega + \beta \left(\frac{\partial \psi}{\partial Y}\frac{\partial \omega}{\partial X} - \frac{\partial \psi}{\partial X} \frac{\partial \omega}{\partial Y} \right)d\mathfrak{D}
\end{array}
\right.
\end{equation}
For our simulation we have considered linear boundary elements, with the specification that we take account of the left and right boundary unit normal for the corner points.

The matrices $\textbf{H}$ and $\textbf{G}$ determine the field contribution of the source node $i$ to the boundary node $j$, defined by
\begin{equation}
H_{ij} = \int_{\Gamma_{j}} u_{q}^{*} d\Gamma, \quad G_{ij} = \int_{\Gamma_{j}} u^{*}d\Gamma
\end{equation} 
The elements $G_{ij} $ represent the flux generated by the source point $i$ around the boundary node $j$. Note that the diagonal values of $\textbf{H}$  require special attention. The $H_{ii}$ values are $c_{i} = \frac{1}{2}$ and $c_{i} = \frac{1}{4}$ for the corner points, since we have employed a linear boundary elements. The diagonal terms of $\textbf{G}$ are computed using MATLAB's integral function.

The Dual Reciprocity boundary element method transforms the domain integral into boundary integrals by expanding the nonhomogenities in terms of radial basis functions $f_{j}'s$ (coordinate functions) of the $N$ boundary nodes and the selected $L$ interior nodes, i.e., 
\begin{align}
-\omega \approx \sum_{j=1}^{N+L} \alpha_{j}f_{j}, \quad
\alpha \omega + \beta \left(\frac{\partial \psi}{\partial Y}\frac{\partial \omega}{\partial X} - \frac{\partial \psi}{\partial X} \frac{\partial \omega}{\partial Y} \right) \approx \sum _{j = 1}^{N+L} \alpha^{*}_{j}f_{j}
\end{align}

The radial basis functions are linked to the particular solution of each equation with the Laplace operator, i.e., $\Delta \hat{u} = f$, where $\hat{u}$ is the particular solution. For our case, we have considered the radial basis functions to be $  f =  r $, since it produces the $F$ matrix with the lowest conditioning number from the commonly used radial basis functions. Introducing these expansions in \eqref{Bem1}, we obtain by the second Green formula the identities
\begin{equation}
\label{Bem3}
\left\{
\begin{array}{ll}
\textbf{H}\psi - \textbf{G}\psi_{q}= \sum_{j=1}^{N+L} \alpha_{j} \left( c_{i}\hat{u}_{ji} + \int_{\Gamma}\psi^{*}_{q}\hat{u}_{j} d\Gamma - \int_{\Gamma}\psi^{*}\hat{q}_{j}  d\Gamma \right)\\ \\
\textbf{H}\omega - \textbf{G}\omega_{q} = \sum_{j=1}^{N+L} \alpha^{*}_{j} \left( c_{i}\hat{u}_{ji} + \int_{\Gamma}\omega^{*}_{q}\hat{u}_{j} d\Gamma - \int_{\Gamma}\omega^{*}\hat{q}_{j}  d\Gamma \right)
\end{array}
\right.
\end{equation}
which can be written in matrix form as the following (see, e.g., \cite{Bozida-2004} and \cite{Gumgum-BE})
\begin{equation}
\label{Bem5}
\left\{
\begin{array}{ll}
\textbf{H}\psi - \textbf{G}\psi_{q}= (\textbf{H}\hat{\textbf{U}} - \textbf{G}\hat{\textbf{Q}}) \textbf{F}^{-1}(-\omega)\\
\textbf{H}\omega - \textbf{G}\omega_{q} = (\textbf{H}\hat{\textbf{U}} - \textbf{G}\hat{\textbf{Q}})\textbf{F}^{-1} \left( \alpha \omega + \beta \left( \frac{\partial \psi}{\partial Y}\frac{\partial \omega}{\partial X} - \frac{\partial \psi}{\partial X} \frac{\partial \omega}{\partial Y} \right) \right),
\end{array}
\right.
\end{equation}
where $\textbf{F}$ is a $(N+L)\times (N+L)$ matrix containing the coordinate functions evaluated at all points. The matrices $\hat{\textbf{U}}$ and $\hat{\textbf{Q}}$ are constructed by taking the corresponding particular solutions and normal derivatives of particular solutions as columns. The nonhomogenities are approximated by the DRBEM idea as
\begin{align}
\label{BemIdea}
\frac{\partial \psi}{\partial X} = \frac{\partial \textbf{F}}{\partial X} \textbf{F}^{-1}\psi,\ \ \frac{\partial \psi}{\partial Y} = \frac{\partial \textbf{F}}{\partial Y} \textbf{F}^{-1}\psi, \ \ 
\frac{\partial \omega}{\partial X} = \frac{\partial \textbf{F}}{\partial X} \textbf{F}^{-1}\omega, \ \ \frac{\partial \omega}{\partial Y} = \frac{\partial \textbf{F}}{\partial Y} \textbf{F}^{-1}\omega.
\end{align}
Applying \eqref{BemIdea} to the system of equations \eqref{Bem5} and making the following notations\footnote{ Note that the multiplication of $\textbf{F}_{y}\textbf{F}^{-1}\psi$ and $\textbf{F}_{x}\textbf{F}^{-1}\psi$ has to be put into a diagonal matrix in order to evaluate \textbf{NonL}} 
\begin{equation}
\textbf{S} = (\textbf{H}\hat{\textbf{U}} - \textbf{G}\hat{\textbf{Q}}) \textbf{F}^{-1}, \ \textbf{NonL} = \beta\left( \textbf{F}_{y}\textbf{F}^{-1}\psi \textbf{F}_{x}\textbf{F}^{-1}\omega - \textbf{F}_{x}\textbf{F}^{-1}\psi \textbf{F}_{y}\textbf{F}^{-1}\omega \right), \nonumber
\end{equation}
we obtain the following system that has to be solved in a recursive way
\begin{equation}
\label{Bem7}
\left\{
\begin{array}{ll}
\textbf{H}\psi - \textbf{G}\psi_{q}= \textbf{S}(-\omega)\\
\textbf{H}\omega - \textbf{G}\omega_{q} = \textbf{S}\left( \alpha \omega + \textbf{NonL} \right).
\end{array}
\right.
\end{equation}
For more details about DRBEM, we refer the reader to the books \cite{Brebbia-Telles} and \cite{Brebbia-Wrobel}.

\subsection{Integration scheme}
In order to increase the numerical stability of the system, we add the time derivative of the vorticity $\partial \omega/ \partial t$ to the second equation of (\ref{Bem7}). This term is discretized by forward difference numerical procedure, which leads to the relation (see, e.g., \cite{Gumgum-BE})
\begin{equation}
\label{dO}
\dfrac{\partial \omega}{\partial t} = \dfrac{\omega^{m+1} - \omega^{m}}{\Delta t}
\end{equation}
The steady state of the flow is obtained at a tolerance error of $tol = 10^{-6}$. In order to accelerate the convergence to steady state and to overcome stability problems a relaxation procedure is used for $\omega$ in the following form
\begin{equation}
\label{rO}
\omega = \gamma \omega^{m+1} + (1 - \gamma)\omega^{m},\quad \omega_{q} = \gamma_{q} \omega_{q}^{m+1} + (1 - \gamma_{q})\omega_{q}^{m}.
\end{equation}
By introducing the relations (\ref{dO}) and (\ref{rO}) into the second equation from (\ref{Bem7}) and rearranging the expression, we will get the final equation of the form:
\begin{align}
\label{finalSimEq}
&\left[ -\dfrac{\textbf{S}}{\Delta t} + \gamma\textbf{H} - \beta  \gamma \textbf{S} \ \textbf{NonL} - \alpha \gamma \textbf{S} \right] \omega^{m+1}  - \gamma_{q} \textbf{G} \omega_{q}^{m+1} =  \\
&\left[ - \dfrac{\textbf{S}}{\Delta t} - (1-\gamma) \textbf{H} + \beta  (1-\gamma)\textbf{S} \ \textbf{NonL} + \alpha (1-\gamma) \textbf{S} \right] \omega^{m}  - (1-\gamma_{q}) \textbf{G} \omega_{q}^{m}. \nonumber 
\end{align}

The following criteria is used to check the convergence of the method
\begin{equation}
\frac{|\psi^{m+1} - \psi^{m}|}{|\psi^{m}|} \le \varepsilon \ \ \mbox{and} \ \quad \frac{|\omega^{m+1} - \omega^{m}|}{|\omega^{m}|} \le \varepsilon
\end{equation}
where $ \varepsilon $ is a prescribed error, which was taken $ 10^{-6} $.

Note that the Dirichlet boundary values for the stream function are known, i.e., they are equal to zero over the whole boundary. The iteration cycle performed for the calculation of the solution is as follows:
\begin{enumerate}
\item[(i)] First, we evaluate the Neumann boundary conditions for the stream function from the last condition in \eqref{Boundary} by the DRBEM approximation \eqref{BemIdea}.
\item[(ii)] By the first relation of \eqref{Bem7}, we obtain the boundary values for the vorticity and the stream function values for the interior nodes.
\item[(iii)] Equation \eqref{finalSimEq}, reduced to a matrix form $\textbf{A}x = \textbf{b}$, leads to the vorticity values for the interior domain and the Neumann boundary values on the frontier.
\item[(iv)] Jump to step $(i)$ until the convergence error is less than $10^{-6}$.
\end{enumerate}

\section{Numerical results and discussions}
\subsection{Numerical stability}
A fundamental role in the stability of the numerical method used in simulations is played by the discretization of the domain (see,.e.g., \cite{Erturk}). In order to investigate the numerical stability of this method, we first consider the Navier-Stokes system ( $\alpha = 0$, $Re = 100$, $\phi = 1$ ) without any sliding on the upper boundary ($S = 0$). Table \ref{T2} shows the dependency of the maximal absolute value of the stream function regarding the number of boundary elements $N$ and the interior nodes $L = K \times K$. The grid size chosen for the simulations is $N = 320$ and $K = 39$, since the last two diagonal iterations from Table \ref{T2} are at a smaller distance than $\varepsilon = 10^{-4}$. In order to obtain smooth plots, we consider a triangulation-based cubic interpolation over the $L$ interior points of $1000 \times 1000$ points.
\begin{table}[h]
\centering
    \begin{tabular}{| c | c | c | c | c | c | c |}
    \hline
    $|\psi_{max}|$  & $N = 200$ & $N = 256$ & $N = 320$ & $N = 360$   \\ \hline
     $K = 24$     & 0.103866 & 0.104028 & 0.103504 & 0.103426      \\ \hline
     $K = 31$     & 0.103959 & 0.103857 & 0.103653 & 0.103567      \\ \hline
     $K = 39$     & 0.103892 & 0.103751 & 0.103638 & 0.103573      \\ \hline
     $K = 44$     & 0.103988 & 0.103845 & 0.103738 & 0.103681      \\ \hline
    \end{tabular}
\caption{Grid dependency: Maximum value of the stream function for the Navier-Stokes system $ |\psi_{max}| $ for $ Re = 100 $, $ \phi = 1 $ and $ S = 0 $.}
\label{T2}
\end{table}

Moreover, the streamfunction values for the Navier-Stokes system are compared to other results in the literature related to the two-dimensional lid-driven fluid flow problem.  Table \ref{TableComp} presents the results of our numerical simulations for different Reynolds numbers, and the results obtained in \cite{Gutt}, \cite{Gupta} and \cite{Marchi}. 
\begin{table}[h]
\small
\centering
    \begin{tabular}{| c | c | c | c |}
    \hline
    $|\psi_{max}|$& $Re = 10$ & $Re = 100$ & $Re = 1000$  \\ 
    (center) & & & \\ \hline
    Present Paper & 0.1001 & 0.1036 & 0.1187    \\ 
    &(0.5175, 0.7658) & (0.6136, 0.7367) & (0.5275, 0.5715) \\ \hline
    Gutt and Gro\c{s}an \cite{Gutt} & 0.1000 & 0.1034 & - \\ 
     & (0.51, 0.77) & (0.615, 0.74) & -  \\ \hline
    Marchi et al. \cite{Marchi} & 0.1001 & 0.1035 & 0.1189  \\ 
    & (0.516, 0.7646) & (0.616, 0.737) & (0.531, 0.565) \\ \hline
    Erturk  et al. \cite{Erturk} & - & 0.1035 & 0.1187 \\ 
    & - & (0.6152, 0.7363) & (0.5313, 0.5645)  \\ \hline
    \end{tabular}
    \caption{Comparison to classical results for the Navier-Stokes system.}
\label{TableComp}
\end{table}

The last stability analysis is a comparison between the numerical results obtained by using a Gauss-Seidel iterative scheme in \cite{Gutt} (see \cite{Hoffmann} for more details) and the simulations using the DRBEM approach shown in Figure \ref{BEM_GS}. The fluid and the porous medium parameters for the nonlinear Darcy-Forchheimer-Brinkman system chosen are $Re = 100$, $Da = 0.25$, $\varphi = 0.2$ and $\Lambda = 1$.

\begin{figure}[h]
\centering
\DeclareGraphicsExtensions{.png}
%\captionsetup{justification=centering}
\begin{minipage}{.32\textwidth}
  \centering
  $ BEM $ mesh \vspace*{-0.3cm}
  \includegraphics[width=\linewidth]{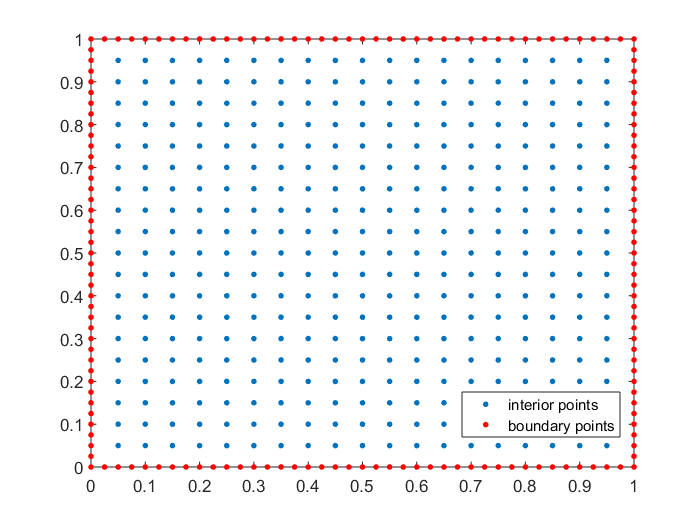}
  $ \bullet $ interior nodes $L$,\\ 
  $ \bullet$ boundary nodes $N$
  \label{fig:sfig14}
\end{minipage}
\begin{minipage}{.32\textwidth}
  \centering
  DRBEM \vspace*{-0.3cm}
  \includegraphics[width=\linewidth]{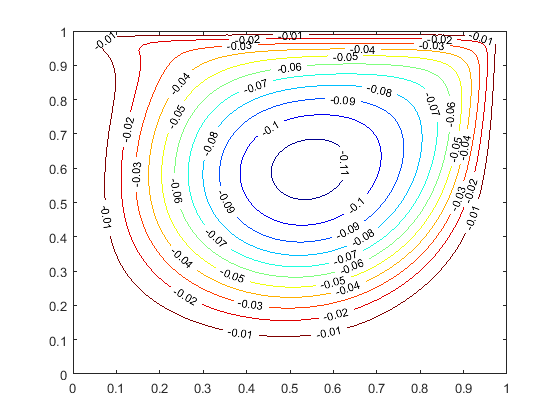}
  $|\psi_{max}| = 0.1132$, \\ 
  $C = (0.5405,0.5603)$
  \label{fig:sfig11}
\end{minipage}%
\begin{minipage}{.32\textwidth}
  \centering
  Gauss-Seidel iteration \vspace*{-0.3cm}
  \includegraphics[width=\linewidth]{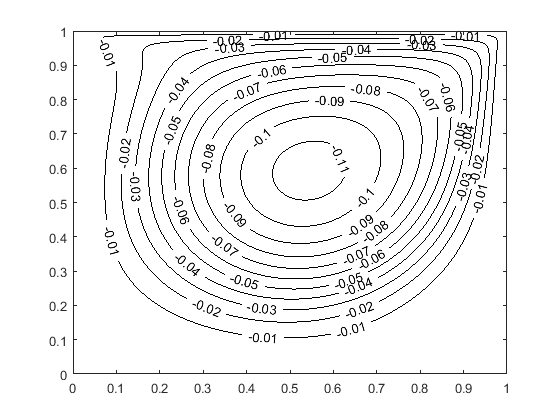}
  $|\psi_{max}| = 0.1139$, \\
  $C = (0.5450,0.5950)$
  \label{fig:sfig12}
\end{minipage}
\caption{ (a) Domain discretization, (b)-(c) Comparison of the fluid flow in a porous medium in the case of the lid-driven problem the Reynolds numbers $ Re=100 $, $Da = 0.25$, $\varphi = 0.2$.}\label{BEM_GS}
\end{figure}
%In order to complete the analysis of the methods involved, we have examined the computational time for the Navier-Stokes and the nonlinear Darcy-Forchheimer-Brinkman system. Table \ref{T3} shows the results for the parameters ($Re = 100$, $Da = 0.25$, $\phi  = 0$) with no sliding condition imposed ($S = 0$).
%\begin{table}[h]
%\centering
%    \begin{tabular}{| c | c | c | c | |}
%    \hline
%    CPU time [s] & Boundary Element Method & Gauss-Seidel iterations \cite{Gutt}  \\ \hline
%     Navier-Stokes     & 398 & 2558   \\ \hline
%     non. D-F-B     & 2258 & 3485     \\ \hline  
%    \end{tabular}
%\caption{Computational time required  }
%\label{T3}
%\end{table}

\subsection{The lid-driven flow in the absence of the sliding parameter}
In the next section, we analyse the structure of the streamlines for some representative fluid parameters ($Re$ and $Da$) in the absence of the sliding parameter, i.e., $S = 0$. Thus we have no sliding of the fluid in the vicinity of the boundary, corresponding to Dirichlet boundary conditions. First, we analyze the dependency of the streamlines for different values of the Reynolds number ($Re = 10, 100, 1000$).  We have set the remaining parameters to the values $\phi = 0.5$, $Da = 0.25$, $\Lambda = 1$. 

\begin{figure}[H]
\centering
%\captionsetup{justification=centering}
\begin{minipage}{.32\textwidth}
  \centering
  $ Re $ = 10 \vspace*{-0.3cm}
  \includegraphics[width=\linewidth]{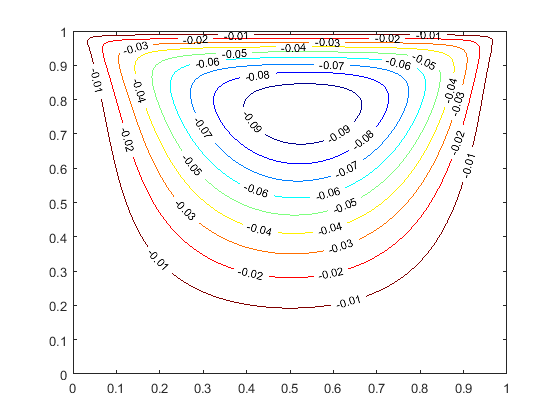}
  $|\psi_{max}|=0.0984$,\\ 
  $C=(0.5315,0.7667)$
  \label{fig:sfig14}
\end{minipage}
\begin{minipage}{.32\textwidth}
  \centering
  $Re=100$ \vspace*{-0.3cm}
  \includegraphics[width=\linewidth]{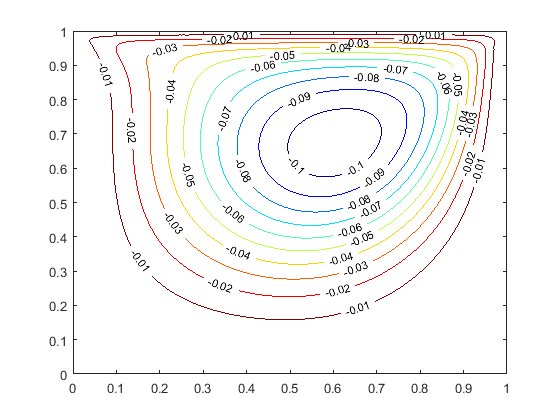}
  $|\psi_{max}|=0.1066$,\\ 
  $C=(0.6036,0.6736)$
  \label{fig:sfig13}
\end{minipage}
\begin{minipage}{.32\textwidth}
  \centering
  $ Re $ = 1000 \vspace*{-0.3cm}
  \includegraphics[width=\linewidth]{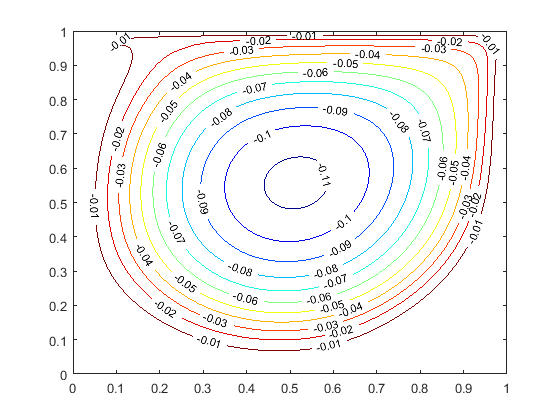}
  $|\psi_{max}|=0.1124$, \\
  $C=(0.5055,0.5675)$
  \label{fig:sfig14}
\end{minipage}
\caption{Comparison of the fluid flow in a porous medium in the case of the lid-driven problem for different Reynolds numbers $ Re= 10 , 100 , 1000 $.}
\label{figRe}
\end{figure}

The Reynolds parameter determines a variation of the geometry of the streamlines, as can be seen from Figure \ref{figRe}. Note that we obtain a decrease of $|\psi_{max}|$, due to the porosity of the medium.

\begin{figure}[h]
\centering
%\captionsetup{justification=centering}
\begin{minipage}{.32\textwidth}
  \centering
  $ Da  = 0.25 $ \vspace*{-0.3cm}
  \includegraphics[width=\linewidth]{SfDFB320K39Re100Da25phi5.png}
  $|\psi_{max}|=0.1066$,\\ 
  $C=(0.6036,0.6736)$
  \label{fig:sfig13}
\end{minipage}
\begin{minipage}{.32\textwidth}
  \centering
  $ Da  = 0.025$ \vspace*{-0.3cm}
  \includegraphics[width=\linewidth]{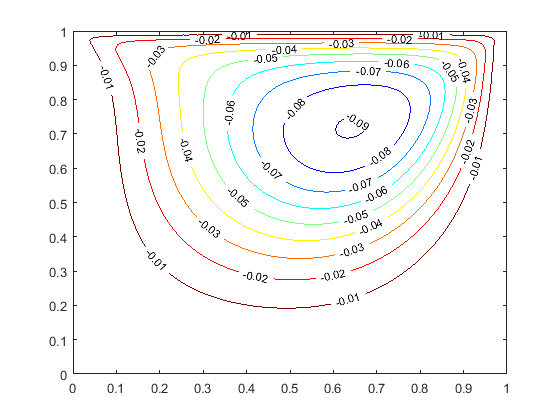}
  $|\psi_{max}|=0.0906$,\\
  $C=(0.6426,0.7207)$
  \label{fig:sfig14}
\end{minipage}
\begin{minipage}{.32\textwidth}
  \centering
  $ Da  = 0.0025$ \vspace*{-0.3cm}
  \includegraphics[width=\linewidth]{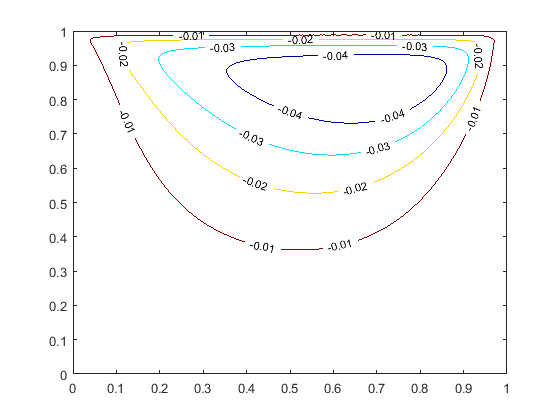}
  $|\psi_{max}|=0.0489$, \\
  $C=(0.692,0.8548)$
  \label{fig:sfig13}
\end{minipage}
\caption{Comparison of the fluid flow in a porous medium in the case of the lid-driven problem for different Darcy numbers $ Da=0.25, 0.025 , 0.0025 $.}
\label{figDa}
\end{figure}

Figure \ref{figDa} shows the structure of the streamlines depending on the Darcy number. It is worth mentioning, that the vortex is moving to the upper right corner of the cavity as the Darcy number decreases and that the strength of the vortex is direct proportional the the Darcy number. This dependency of the fluid flow is consistent with the physical behavior of a fluid.

\subsection{The lid-driven flow in the case of a non-vanishing sliding parameter}
This section is concerned with the mixed Dirichlet-Robin boundary value problem associated to the nonlinear Darcy-Forchheimer-Brinkman system, i.e., we consider an additional sliding parameter imposed on the upper moving wall. A numerical study of the Navier slip condition can be found in \cite{He}. This condition implies that not the whole fluid located in the vicinity of the boundary is driven by the moving wall.

\begin{figure}[ht!]
\centering
%\captionsetup{justification=centering}
\begin{minipage}{.32\textwidth}
  \centering
  $ S= 0 $ \vspace*{-0.3cm}
  \includegraphics[width=\linewidth]{SfDFB320K39Re100Da25phi5.png}
  $|\psi_{max}|=0.1066$, \\ 
  $C=(0.6036,0.6736)$
  \label{fig:sfig14}
\end{minipage}
\begin{minipage}{.32\textwidth}
  \centering
  $S=0.01$ \vspace*{-0.3cm}
  \includegraphics[width=\linewidth]{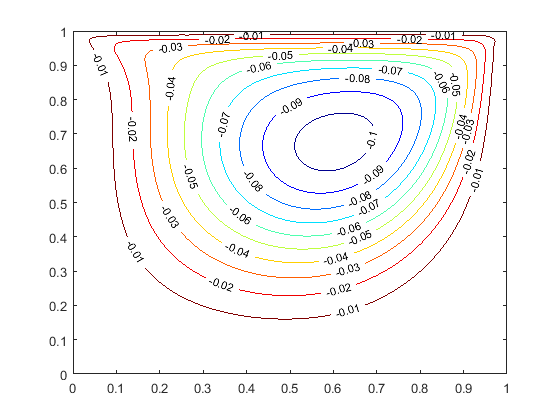}
  $|\psi_{max}|=0.1047$,
  $ C=(0.6066,0.6766)$
  \label{fig:sfig13}
\end{minipage}
\begin{minipage}{.32\textwidth}
  \centering
  $ S = 0.1 $ \vspace*{-0.3cm}
  \includegraphics[width=\linewidth]{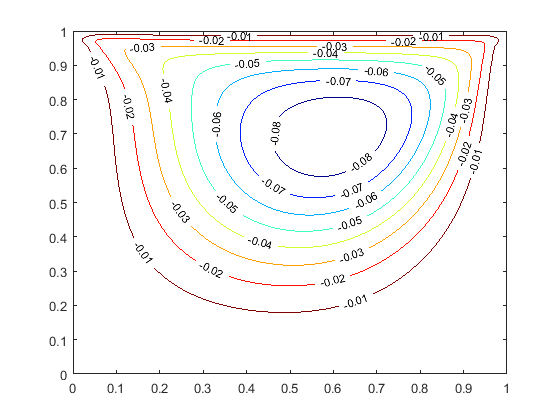}
  $|\psi_{max}|=0.0882,$  
  $C = ( 0.596,0.695)$
  \label{fig:sfig14}
\end{minipage}
\caption{Comparison of the fluid flow in a porous medium in the case of the lid-driven problem for $ S= 0 , 0.01 , 0.1 $.}
\label{TableS}
\end{figure}

Figure \ref{TableS} shows that the structure of the streamlines of the fluid flow change slightly with the increasing sliding parameter and also the strength of the maximal value of the stream function decreases as we increase the sliding parameter, which is in good agreement with the physical meaning of the parameter.

\subsection{The lid-driven flow for the upper wall partitioned in two segments}
Theorem \ref{Theorem 3.2} states that the nonlinear Darcy-Forchheimer-Brinkman system in a two-dimensional Lipschitz domain is well-posed even if the moving wall of the square cavity is partitioned in two opposite moving parts $\Gamma^{+}_{4} := [0,0.5] \times 1$ and $\Gamma_{4}^{-}:= [0.5,1] \times 1$ of same length, which would correspond to the boundary conditions
\begin{equation}
\frac{\partial \Psi}{\partial Y} = -1 \quad \mbox{on} \ \Gamma^{-}_{4}, \quad  \ \frac{\partial \Psi}{\partial Y} = 1 + S \frac{\partial^{2} \Psi}{\partial Y^{2}} \quad \mbox{on} \ \Gamma^{+}_{4}.
\end{equation}
The intersection point of the two boundaries $\Gamma^{+}_{4} \cap \Gamma^{-}_{4}$ represents a measure zero subset on which either Dirichlet or Robin boundary conditions are not defined, since a perfect physical contact between the two partitions cannot be obtained. Clearly the sliding parameter $S$  satisfies the condition \eqref{lambda}. Both segments move with the same velocity in opposite directions, but the slip condition $S \neq 0$ is imposed only on the upper left wall.
\begin{figure}[h]
\centering
%\captionsetup{justification=centering}
\begin{minipage}{.32\textwidth}
  \centering
  $ \xrightarrow{\ \ S = 0 \ \ } \qquad \xleftarrow{ \ \ S = 0 \ \ }$ \vspace{-0.3cm}
  \includegraphics[width=\linewidth]{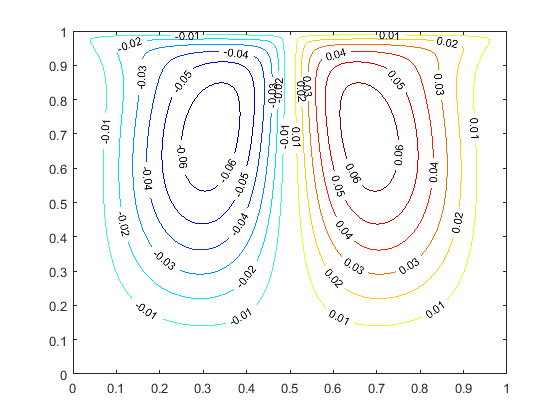}
  $\psi_{min}=-0.0668$, $C_{min}=(0.320,0.696)$,\\ 
  $\psi_{max}=0.0668$, $C_{max}=(0.679,0.696)$
  \label{fig:sfig14}
\end{minipage}
\begin{minipage}{.32\textwidth}
  \centering
  $ \xrightarrow{\ \  S = 0 \ \ } \quad \xleftarrow{  S = 0.01  }$ \vspace{-0.3cm}
  \includegraphics[width=\linewidth]{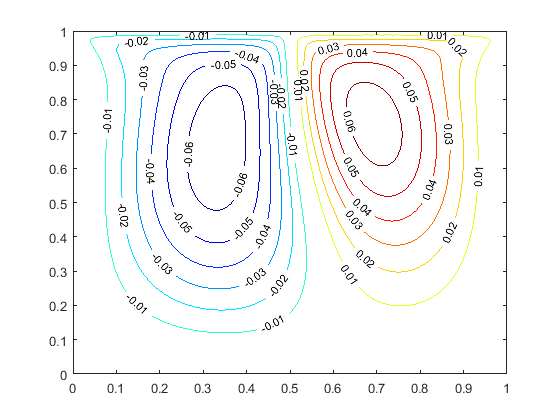}
  $\psi_{min}=-0.0669$, $C_{min}=(0.335,0.656)$, \\
  $\psi_{max}=0.0659$, $C_{max}=(0.692,0.733)$
  \label{fig:sfig13}
\end{minipage}
\begin{minipage}{.32\textwidth}
  \centering
  $ \xrightarrow{\ \  S = 0 \ \ } \quad \xleftarrow{ \ S = 0.1 \ \ }$ \vspace{-0.3cm}
  \includegraphics[width=\linewidth]{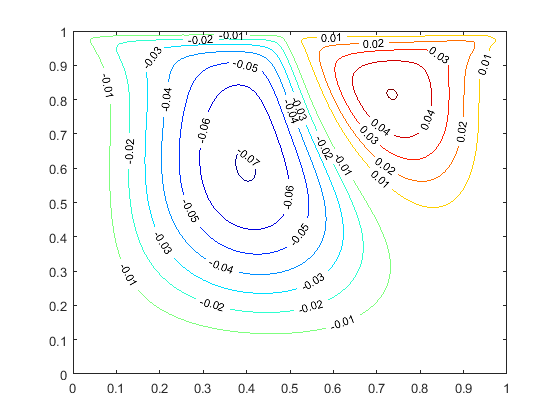}
  $\psi_{min}=-0.0705$, $C_{min}=(0.398,0.608)$, \\
  $\psi_{max}=0.0501$, $C_{max}=(0.734,0.815)$
  \label{fig:sfig14}
\end{minipage}
\caption{Comparison of the fluid flow in a porous medium in the case of the lid-driven problem for $ S= 0 , 0.01 , 0.1 $.}
\label{fig2S}
\end{figure}

Figure \ref{fig2S} show the importance of the additional sliding parameter, because the stronger vorticity in given by the boundary segment without sliding condition and the corresponding vortex fills the main part of the cavity. 

\section{Conclusions}
Boundary integral equations (BIE's) are closely related to Boundary Element Methods, such as DRBEM. Although, we derive the BIE for the mixed Dirichlet-Neumann boundary value problem for the homogeneous Brinkman system, the well-posedness result is extended by means of potential theory to the Poisson problem for the Brinkman system and to Dirichlet-Robin boundary conditions using the linearity of the solution operator. Having these results, we give a well-posedness result of the weak solution for the nonlinear Darcy-Forchheimer-Brink\-man system, by showing that the nonlinear operator describing the system has a fixed point using Banach-Caciopolli fixed point theorem, under the assumption of small boundary values.

The lid driven cavity flow problem is considered for the purpose of giving an application to the well-posedness results given in the first part and in order to use a Boundary Element Method (DRBEM) for the computation of such a solution. Since the Lipschitz domain considered is two-dimensional, it is convenient to introduce the steam function and the vorticity, obtaining Poisson equations in both functions, which can be solved in an iterative way considering the fundamental solution of the Laplacian. 

%The main advantage of DRBEM is the lower computational time required in comparison to the Gauss-Seidel iterative scheme.

The first order perturbation of the Navier-Stokes system through the Darcy coefficient describes the motion of an incompressible fluid located in a saturated porous domain. The physical properties of such a fluid are described in the case of the lid driven cavity flow problem for some representative value of the Reynolds number $Re$, the Darcy number $Da$ and the sliding parameter $S$, for a constant porosity $\phi = 0.5$. Moreover, in order to highlight the importance of the sliding parameter imposed on the moving wall, we consider the moving wall partitioned into two opposite moving parts.

\section*{Acknowledgement}
Part of this work was done in October-November 2016 during a visit at the Department of Mathematics of the
University of Padova. The author is grateful to the members of this department and specially to Prof. Massimo Lanza de Cristoforis for their hospitality.

%% If you have bibdatabase file and want bibtex to generate the
%% bibitems, please use
%%
\bibliographystyle{plain}

%% else use the following coding to input the bibitems directly in the
%% TeX file.

\end{document}